\documentclass[10pt]{amsart}
\usepackage[
paper=a4paper,
text={147mm,230mm},centering
]{geometry}

\iftrue
\makeatletter
\def\@settitle{%
  \vspace*{-20pt}
  \begin{flushleft}%
    \baselineskip14\p@\relax
    \normalfont\bfseries\LARGE
    \@title
  \end{flushleft}%
}
\def\@setauthors{%
  \begingroup
  \def\thanks{\protect\thanks@warning}%
  \trivlist
  \large \@topsep30\p@\relax
  \advance\@topsep by -\baselineskip
  \item\relax
  \author@andify\authors
  \def\\{\protect\linebreak}%
  \authors
  \ifx\@empty\contribs
  \else
    ,\penalty-3 \space \@setcontribs
    \@closetoccontribs
  \fi
  \normalfont
  \endtrivlist
  \endgroup
}
\def\@setabstracta{%
    \ifvoid\abstractbox
  \else
    \skip@25\p@ \advance\skip@-\lastskip
    \advance\skip@-\baselineskip \vskip\skip@
    \box\abstractbox
    \prevdepth\z@ 
    \vskip-10pt
  \fi
}
\renewenvironment{abstract}{%
  \ifx\maketitle\relax
    \ClassWarning{\@classname}{Abstract should precede
      \protect\maketitle\space in AMS document classes; reported}%
  \fi
  \global\setbox\abstractbox=\vtop \bgroup
    \normalfont\small
    \list{}{\labelwidth\z@
      \leftmargin0pc \rightmargin\leftmargin
      \listparindent\normalparindent \itemindent\z@
      \parsep\z@ \@plus\p@
      
    }%
    \item[\hskip\labelsep\bfseries\abstractname.]%
}{%
  \endlist\egroup
  \ifx\@setabstract\relax \@setabstracta \fi
}
\def\section{\@startsection{section}{1}%
  \z@{-1.2\linespacing\@plus-.5\linespacing}{.8\linespacing}%
  {\normalfont\bfseries\large}}
\def\subsection{\@startsection{subsection}{2}%
  \z@{-.8\linespacing\@plus-.3\linespacing}{.3\linespacing\@plus.2\linespacing}%
  {\normalfont\bfseries}}
\def\subsubsection{\@startsection{subsubsection}{3}%
  \z@{.7\linespacing\@plus.1\linespacing}{-1.5ex}%
  {\normalfont\itshape}}
\def\@secnumfont{\bfseries}
\makeatother
\fi 

\usepackage{amssymb}
\usepackage{mathrsfs}
\usepackage[all]{xy}
\usepackage{graphicx,color,float}
\usepackage[bookmarks=false]{hyperref}
\usepackage{pinlabel}
\usepackage[textwidth=1.3in,color=yellow]{todonotes}
\usepackage{amsmath}
\usepackage{amsmath}
\usepackage{amsfonts}
\usepackage{graphicx}
\usepackage{amscd}
\usepackage{url}
\usepackage{caption}
\usepackage{subcaption}
\usepackage{color}
\usepackage{stmaryrd}

\theoremstyle{plain}
\newtheorem{theorem}{Theorem}[section]
\newtheorem{thmx}{Theorem}

\newtheorem{proposition}[theorem]{Proposition}
\newtheorem{lemma}[theorem]{Lemma}
\newtheorem{corollary}[theorem]{Corollary}

\newtheorem{hypothesis}[theorem]{Hypothesis}

\theoremstyle{definition}

\newtheorem{definition}[theorem]{Definition}
\newtheorem{example}[theorem]{Example}

\theoremstyle{remark}
\newtheorem{remark}[theorem]{Remark}

\newcommand{\C}{\mathbb{C}}

\newcommand{\N}{\mathbb{N}}
\newcommand{\Q}{\mathbb{Q}}
\newcommand{\R}{\mathbb{R}}

\newcommand{\Z}{\mathbb{Z}}

\def\pa{\partial}
\def\mcal{\mathcal}
\def\frak{\mathfrak}
\def\scr{\mathscr}
\def\ev{\textup{ev}}
\def\bev{\textbf{\textup{ev}}}

\numberwithin{equation}{section}

\begin{document}

\title[Landau--Ginzburg
potentials with bulk in Gelfand--Cetlin systems]{A critical point analysis of Landau--Ginzburg
potentials with bulk in Gelfand--Cetlin systems}

\author{Yunhyung Cho}
\address{Department of Mathematics Education, Sungkyunkwan University, Seoul, Republic of Korea}
\email{yunhyung@skku.edu}

\author{Yoosik Kim}
\address{Department of Mathematics, Brandeis University, Waltham, USA and Center of Mathematical Sciences and Applications, Harvard University, Cambridge, USA}
\email{yoosik@brandeis.edu, yoosik@cmsa.fas.harvard.edu}

\author{Yong-Geun Oh}
\address{Center for Geometry and Physics, Institute for Basic Science (IBS),  Pohang, Republic of Korea, and Department of Mathematics, POSTECH, Pohang, Republic of Korea}
\email{yongoh1@postech.ac.kr}

\thanks{The first named author was supported by the National Research Foundation of Korea(NRF) grant funded by the Korea government(MSIP; Ministry of Science, ICT \& Future Planning) (NRF-2017R1C1B5018168).
The second named author is supported by the Simons Collaboration Grant on
Homological Mirror Symmetry and Applications. The third named author is supported by the IBS project IBS-R003-D1.}

\begin{abstract}
Using the bulk-deformation
of Floer cohomology by Schubert cycles and non-Archimedean analysis of Fukaya--Oh--Ohta--Ono's
bulk-deformed potential function, we prove that
every complete flag manifold $\mathrm{Fl}(n)$ ($n \geq 3$) with a monotone Kirillov--Kostant--Souriau
symplectic form carries a continuum of non-displaceable Lagrangian tori which degenerates
to a non-torus fiber in the Hausdorff limit. In particular, the Lagrangian $S^3$-fiber in
$\mathrm{Fl}(3)$ is non-displaceable, answering the question of which was raised by Nohara--Ueda
who computed its Floer cohomology to be vanishing.
\end{abstract}
\maketitle
\setcounter{tocdepth}{1}
\tableofcontents

\section{Introduction}\label{Sec_Intropart2}

This article is a sequel to Part I of \cite{CKO}. Here we focus on detecting non-displaceable Lagrangian
fibers of the Gelfand--Cetlin systems. We abbreviate `Gelfand--Cetlin' by writing
`GC' and so on.

Lagrangian Floer theory on GC systems was firstly studied by Nishinou, Nohara, and Ueda \cite{NNU}.
They proved that the Lagrangian GC torus fiber at the center of the GC polytope is non-displaceable. To show it, they calculated the potential function by constructing a toric degeneration from a GC system to a toric moment map and finding a weak bounding cochain such that the deformed Floer cohomology is non-vanishing.
The GC system admits non-torus Lagrangian GC fibers at the lower dimensional strata of the GC polytope, which make Floer theory of the system more interesting and challenging.
Using non-abelian symmetry or discrete symmetry, particular fibers of limited cases of Grassmannians have been investigated in \cite{EL, NU, ELgF}.

Motivated by those works, the present authors systematically studied topological/geometrical
type of GC fibers and provided a complete classification thereof in terms of the combinatorial ladder diagram in \cite{CKO}.
In this paper, we discuss non-displaceable GC fibers on complete flag manifolds equipped with monotone Kirillov--Kostant--Souriau symplectic forms (KKS forms).

Let $\lambda = ( \lambda_1 > \lambda_2 > \cdots > \lambda_n )$
be a decreasing sequence of real numbers so that the co-adjoint orbit $\mathcal{O}_\lambda$ is diffeomorphic to a complete flag manifold $\mathrm{Fl}(n)$.
When the KKS form $\omega_\lambda$ is \emph{monotone}, we shall show non-displaceability of certain torus fibers and non-torus fibers.
First, we describe putative positions of non-displaceable fibers.
By scaling $\omega_\lambda$ if necessary, we may assume that
\begin{equation}\label{givensequencemonotone}
	\lambda = ( \lambda_{i} := n - 2i + 1 \mid i = 1, \cdots, n ),
\end{equation}
which is the case where $[\omega_\lambda] = c_1(T\mathcal{O}_\lambda)$. In this case, the GC polytope $\Delta_\lambda$ is a reflexive polytope so that it admits a unique lattice point in its interior, see \cite[Corollary 2.2.4]{BCKV}.
Each Lagrangian face, a face containing Lagrangian fibers in its relative interior, admits a unique point satisfying certain Bohr--Sommerfeld condition, at which a monotone Lagrangian fiber is located \cite{CK:mono}.
A candidate is a line segment connecting the center of the polytope and the position of a monotone Lagrangian fiber (in the case of monotone complete flag manifolds).

To verify a non-torus Lagrangian fiber is non-displaceable, a family of Lagrangian tori whose Hausdorff limit is the non-torus fiber will be taken into account.
Once we show that the tori are non-displaceable, we then obtain non-displaceability of the non-torus Lagrangian fiber.

We start from the simplest case where the co-adjoint orbit $\mcal{O}_\lambda$ of a sequence $\lambda = ( \lambda_ 1 > \lambda_2 > \lambda_3 )$. In this case, Pabiniak investigated displaceable GC fibers.

\begin{theorem}[\cite{Pa}]\label{displaceablefl3}
For $\lambda = ( \lambda_ 1 > \lambda_2 > \lambda_3 )$, let $(\mcal{O}_\lambda, \omega_\lambda)$ be
as above.
\begin{enumerate}
\item If $\omega_\lambda$ is not monotone, all the fibers but one over the center are displaceable.
\item If $\omega_\lambda$ is monotone, all the fibers but those over the line segment
\begin{equation}\label{equation_linesegmentforf3}
	I := \left\{ (u_{1,1}, u_{1,2}, u_{2,1}) = (0, a -t, -a +t) \in \R^3 ~\colon~ 0 \leq t \leq a \right\}
\end{equation}
are displaceable where $2a = \lambda_1 - \lambda_2$. Observe the line segment $I$ is the red line in Figure~\ref{figure_GCmomnetf3}.
\end{enumerate}
\end{theorem}

\begin{figure}[h]
	\scalebox{1.7}{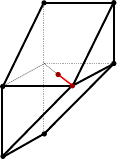}
	\caption{\label{figure_GCmomnetf3} The positions of non-displaceable GC Lagrangian fibers in $\mathrm{Fl}(3)$.}	
\end{figure}
\vspace{-0.2cm}

Note that the line segment $I$ connects the center $(0, a, -a)$ of the GC polytope $\Delta_\lambda$ and the position $(0, 0, 0)$ of the Lagrangian $3$-sphere.

The following theorem asserts every GC fiber in the family $\{\Phi_\lambda^{-1}(\mathbf{u}) \mid \mathbf{u} \in I\}$ is indeed non-displaceable.
Thus, together with Theorem~\ref{displaceablefl3}, our result provides the complete classification of displaceable and non-displaceable Lagrangian GC fibers when $\omega_\lambda$ is monotone.

\begin{thmx}[Theorem \ref{theorem_maincompleteflag3}]\label{theoremC}
Let $\lambda = ( \lambda_1 = 2 > \lambda_2 = 0 > \lambda_3 = -2 )$ and
consider the co-adjoint orbit $(\mcal{O}_\lambda, \omega_\lambda)$.
 Then the fiber over a point $\mathbf{u} \in \Delta_\lambda$ is non-displaceable if and only if $\mathbf{u} \in I$ where
\begin{equation}\label{linesegmentforf3}
I := \left\{ (u_{1,1}, u_{1,2}, u_{2,1}) = (0, 1 -t, -1 +t) \in \R^3 ~|~ 0 \leq t \leq 1 \right\}
\end{equation}
In particular, the Lagrangian 3-sphere $\Phi_\lambda^{-1}(0,0,0)$ is non-displaceable.
\end{thmx}

\begin{remark}
Nohara and Ueda \cite{NU} calculated a Floer cohomology of the Lagrangian 3-sphere $\Phi^{-1}_\lambda(0,0,0)$, which turns out to be zero over the Novikov field $\Lambda$ and hence
non-displaceabilty of the fiber remained open. Theorem~\ref{theoremC} resolves
the question showing that it is non-displaceable.
\end{remark}

Next, we deal with a general case for an arbitrary positive integer $n \geq 4$ where $\lambda$ is given as in~\eqref{givensequencemonotone}. In this case, the GC polytope $\Delta_\lambda$ is a reflexive polytope whose center is
\[
\left( u _{i, j} := j - i \mid {i + j \leq n} \right) \in \Delta_\lambda \subset \R^{n(n-1)/2}.
\]
Consider the face $f_m$ of $\Delta_\lambda$ defined by
\[
\{ u_{i,j} = u_{i, j+1} \mid  1 \leq i \leq m, 1 \leq j \leq m-1\} \cup \{ u_{i+1,j} = u_{i, j} \mid  1 \leq i \leq m-1, 1 \leq j \leq m\}
\]
for any integer $m$ satisfying $2 \leq m \leq \left\lfloor \frac{n}{2} \right\rfloor$.
Note that there are $\left( \left\lfloor \frac{n}{2} \right\rfloor - 1 \right)$ such faces in $\Delta_\lambda$.

The candidates for non-displaceable Lagrangian fibers are the fibers over the line segment $I_m \subset \Delta_\lambda$
connecting the center of $\Delta_\lambda$ and the center
of $f_m$ for each $m \geq 2$.
Explicitly,
the line segment $I_m$ is parameterized by $\{I_m(t) \in \Delta_\lambda \mid 0 \leq t \leq 1\}$ where
\begin{equation}\label{IMT}
I_m (t) :=
\begin{cases}
u_{i,j}(t) = (j - i) - (j - i) \, t \quad &\mbox{if \,} \max (i,j) \leq m \\
u_{i,j}(t) = (j - i) \quad &\mbox{if \,} \max (i,j) > m.
\end{cases}
\end{equation}
We denote by $L_m(t)$ the Lagrangian GC fiber over the point $I_m(t)$, that is $L_m(t) := \Phi_\lambda^{-1}(I_m(t))$.
The following is the main theorem of the present paper.

\begin{thmx}[Theorem ~\ref{theorem_completeflagmancotinuum}]\label{theoremD}
Let $\lambda = ( \lambda_{i} := n - 2i + 1 \,|\, i = 1, \cdots, n )$ be an $n$-tuple of real numbers for an arbitrary integer $n \geq 4$. Consider the co-adjoint orbit $(\mathcal{O}_\lambda,\omega_\lambda)$.
Then each Gelfand--Cetlin fiber $L_m(t)$ is non-displaceable Lagrangian for every $2 \leq m \leq \left\lfloor \frac{n}{2} \right\rfloor$.
\end{thmx}

The main ingredient of the proof of the theorem is a careful non-Archimedean
(or $T$-adic) analysis of
 the critical point equation of the bulk-deformed potential in the spirit of \cite{FOOOToric2,FOOOS2S2,KLS} generalizing the analysis from the toric to the Gelfand--Cetlin system.

\begin{example}
A monotone complete flag manifold $\mathrm{Fl}(7)$ admits (at least) two line segments $I_2$ and $I_3$ in the GC polytope over which the fibers are non-displaceable as depicted in Figure~\ref{figureF7}. Particularly, it has non-displaceable Lagrangian fibers diffeomorphic to $\mathrm{U}(2) \times T^{17}$ and $\mathrm{U}(3) \times T^{12}$.
\begin{figure}[ht]
	\scalebox{1}{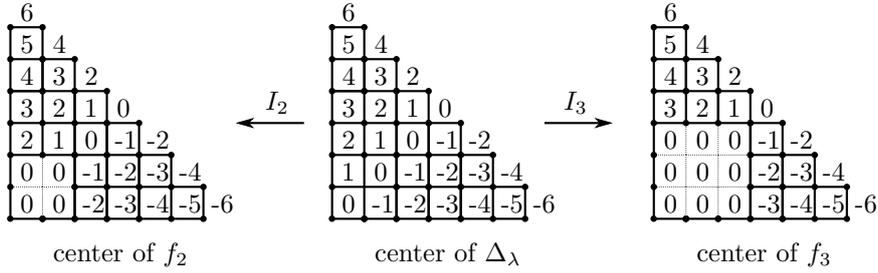}
	\caption{\label{figureF7} Positions of non-displaceable Lagrangian GC fibers in $\mathrm{Fl}(7)$.}	
\end{figure}
\end{example}

\begin{remark}
	The third named author with Fukaya, Ohta, and Ono \cite{FOOOToric2} found a continuum of non-displaceable torus fibers on some compact toric manifolds including a non-monotone toric blowup of $\C P^2$ at two points, see also Woodward \cite{Wo}. Using the degeneration models, they also produced a continuum of non-displaceable Lagrangian tori on $\C P^1 \times \C P^1$ and the cubic surface respectively in \cite{FOOOS2S2} and \cite{FOOOsp}. Vianna \cite{Vi} also showed a continuum of non-displaceable tori in $(\C P^1)^{2n}$.
Sun \cite{Sun} found non-displaceable Lagrangian tori near a chain of Lagrangian two-spheres in a closed symplectic four manifold and del Pezzo surfaces.
\end{remark}

\section{Lagrangian Floer theory on Gelfand--Cetlin systems}\label{secReviewOfLagrangianFloerTheory}

In this section, after briefly recalling Lagrangian Floer theory and its deformation developed by the third named author with Fukaya, Ohta, and Ono in a general context, we review the work of Nishinou, Nohara, and Ueda about the calculation of the potential function of a GC system. Then, using the combinatorial description of Schubert cycles in complete flag manifolds by Kogan, we will compute the potential function deformed by a combination of Schubert cycles of codimension two as a Laurent series. Finally, combining those ingredients, we give the proof of Theorem~\ref{theoremC}.

\subsection{Potential functions of Gelfand--Cetlin systems}\label{reviewpotentialpotentialfunc}

We begin by setting up the notations.
\begin{itemize}
\item $\Lambda := \left\{ \sum_{i=1}^{\infty} a_{i} T^{\lambda_i}  \mid a_i \in \C, \lambda_i \in \R, \lim_{i \rightarrow \infty} \lambda_i = \infty \right\}$
\item $\frak{v}_T \colon \Lambda \backslash \{ 0 \} \to \R, \quad \frak{v}_T \left(\sum_{j=1}^{\infty} a_{j} T^{\lambda_j}  \right) := \inf_{j} \, \{ \lambda_j \mid a_j \neq 0 \}$
\item $\Lambda_0 := \frak{v}_T^{-1}[0, \infty) \cup \{ 0 \} = \left\{ \sum_{i=1}^{\infty} a_{i} T^{\lambda_i} \in \Lambda  \mid \lambda_i \geq 0 \right\}$
\item $\Lambda_+ := \frak{v}_T^{-1}(0, \infty) \cup \{ 0 \} = \left\{ \sum_{i=1}^{\infty} a_{i} T^{\lambda_i} \in \Lambda  \mid \lambda_i > 0 \right\}$
\item $\Lambda_U := \Lambda_0 \backslash \Lambda_+ = \left\{ \sum_{i=1}^{\infty} a_{i} T^{\lambda_i} \in \Lambda_0  \mid \frak{v}_T \left( \sum_{i=1}^{\infty} a_{i} T^{\lambda_i} \right) = 0 \right\}.$
\end{itemize}

For a relatively spin closed Lagrangian submanifold $L$ in a closed symplectic manifold $(X, \omega)$, by the work of Fukaya \cite{Fuk}, one can associate an $A_\infty$-algebra  $\{ \frak{m}_k \}_ {k \geq 0}$ on the $\Lambda_0$-valued de Rham complex of $L$.
Following the procedure in \cite{FOOOc}, the constructed $A_\infty$-algebra can be converted into the canonical model on $H^\bullet (L; \Lambda_0)$.

A solution $b \in H^1(L; \Lambda_+)$ of the (weak) Maurer--Cartan equation
$$
\sum_{k=0}^\infty \frak{m}_k (b^{\otimes k}) \equiv 0 \mod \textup{PD}[L]
$$
is called a \emph{(weak) bounding cochain}. The value of the \emph{potential function} $W$ at a bounding cochain $b$ is assigned to be the multiple of the Poincar\'{e} dual $\textup{PD}[L]$ of $L$.
Namely, we solve
$$
\sum_{k=0}^\infty \frak{m}_k (b^{\otimes k}) = W(b) \cdot \textup{PD}[L],\,  \mbox{ with $W(b) \in \Lambda_0$}.
$$
Since $\textup{PD}[L]$ is the strict unit of the $A_\infty$-algebra (on the de Rham model), the deformed map
$\frak m^b_1$ defined by
$$
\frak{m}^{b}_1 (h) := \sum_{l,k} \frak{m}_{l+k+1}(b^{\otimes l}, h, b^{\otimes k})
$$
becomes a differential and thus its cohomology (deformed by $b$) over $\Lambda_0$ is defined.
Let
\begin{itemize}
\item ${HF} ((L, b); \Lambda_0) := \textup{Ker} (\frak{m}^b_1) \, / \, \textup{Im} (\frak{m}^b_1)$
\item ${HF} ((L, b); \Lambda) := {HF} ((L, b); \Lambda_0) \otimes_{\Lambda_0} \Lambda.$
\end{itemize}
The reader is referred to \cite{FOOO, FOOOToric1, FOOOToric2, FOOOToric3} for details.

Now, we restrict to the case of a Lagrangian GC torus fiber $\Phi_\lambda^{-1}(\mathbf{u})$ in a co-adjoint orbit $\mcal{O}_\lambda$. As a deformation of the action of the Borel subgroup, Kogan and Miller \cite{KoM} realized the toric degeneration of a flag manifold given by Gonciulea--Lakshmibai \cite{GL}.
Using the degeneration of a (partial) flag manifold to the GC toric variety, Nishinou--Nohara--Ueda constructed a degeneration of the GC system of $\mcal{O}_\lambda$ to the moment map of the toric variety.

\begin{theorem}[Theorem 1.2 in \cite{NNU}]\label{NNUtoricdeg}
For any non-increasing sequence $\lambda = ( \lambda_1 \geq \dots \geq \lambda_n )$, there exists a toric degeneration of the Gelfand--Cetlin system $\Phi_\lambda$ on the co-adjoint orbit $(\mcal{O}_\lambda, \omega_\lambda)$ in the following sense.
\begin{enumerate}
\item There is a flat family $f \colon \mcal{X} \to I = [0,1]$ of algebraic varieties and a two form $\widetilde{\omega}$ on $\mcal{X}$ such that
\begin{enumerate}
\item ${X}_0 := f^{-1}(0)$ is the toric variety associated with the Gelfand--Cetlin polytope $\Delta_\lambda$ and $\omega_0 := \widetilde{\omega}|_{X_0}$ is a torus-invariant K{\"a}hler form.
\item ${X}_1 := f^{-1}(1)$ is the co-adjoint orbit $\mcal{O}_\lambda$ and $\omega_1 = \widetilde{\omega}|_{X_1}$ is the Kirillov--Kostant--Souriau symplectic form {$\omega_\lambda$}.
\end{enumerate}
\item There is a family $\{\Phi_t \mid X_t \rightarrow \Delta_\lambda\}_{0 \leq t \leq 1}$ of completely integrable systems such that $\Phi_0$ is the moment map for the torus action on $X_0$ and $\Phi_1$ is the Gelfand--Cetlin system.
\item Let $\Delta^{\textup{sm}}_\lambda := \Delta_\lambda \backslash \Phi_0 \left( \textup{Sing} (X_0) \right)$ and
$X^{\textup{sm}}_t := \Phi_t^{-1}(\Delta^{\textup{sm}}_\lambda )$ where $\textup{Sing} (X_0)$ is the set of singular points of $X_0$.
Then there exists a flow $\phi_t$ on $\mcal{X}$ such that for each $0 \leq t \leq s$,
the restricted flow $\phi_t |_{X^{\textup{sm}}_{s}} \colon X^{\textup{sm}}_{s} \to X^{\textup{sm}}_{s-t}$ respects the symplectic structures and the complete integrable systems:
\[
	\xymatrix{
		  (X^{\textup{sm}}_{s}, \omega_s)  \ar[dr]_{\Phi_s} \ar[rr]^{ \phi_t|_{X^{\textup{sm}}_{s}} }  & & (X^{\textup{sm}}_{s-t}, \omega_{s-t})
      \ar[dl]^{\Phi_{s-t}} \\
  & \Delta^{\textup{sm}}_\lambda &}
\]
\end{enumerate}
\end{theorem}
Let $\phi^\prime_s \colon X_s \to X_0$ be a (continuous) extension of the flow $\phi_s \colon X^{\textup{sm}}_s \to X^{\textup{sm}}_0$ in Theorem~\ref{NNUtoricdeg} (\cite[Section 8]{NNU}). The extended map $\phi^\prime_s$ transports Floer theory data from the toric moment map to a nearby integrable system.
As the deformation in Theorem~\ref{NNUtoricdeg} is a family of Fano varieties and the GC toric variety admits a small resolution at the singular loci, any holomorphic discs bounded by $L$ intersecting the loci collapsing to the singular loci of $X_0$ must have the Maslov index strictly greater than two so that such discs do \emph{not} contribute to the potential function. Furthermore, because the Fredholm regularity is an open condition, the holomorphic discs of Maslov index two intersecting the toric divisor in the toric variety $X_0$ give rise to regular holomorphic discs at $X_s$ for sufficiently small $s$. Set $X := X_s$ and let $L$ be any Lagrangian torus fiber of $\Phi_s \colon X_s \to \Delta_\lambda$. Combining those with the results of Cho--Oh \cite{CO}, Nishinou--Nohara--Ueda proved the followings.

\begin{enumerate}
\item Each Lagrangian torus fiber $L$ does not bound any non-constant holomorphic discs whose classes are of Maslov index less than or equal to zero. 
\item Every class $\beta \in \pi_2 (X, L)$ of Maslov index two is Fredholm regular. 
\item There is a one-to-one correspondence between the holomorphic discs of Maslov index two bounded by a Lagrangian GC torus fiber and the facets of GC polytope. 
\item For each class $\beta$ of Maslov index 2, the open Gromov--Witten invariant $n_\beta$, which counts the holomorphic discs passing through a generic point in $L$ and representing $\beta$, is $1$. 
\end{enumerate}

The above conditions imply that every $1$-cochain in $H^1(L; \Lambda_0)$ is a weak bounding cochain and hence the potential function can be defined on $H^1(L; \Lambda_0)$. To extend deformation space from $H^1(L; \Lambda_+)$ to $H^1(L; \Lambda_0)$, one twists Floer theory by the holonomy of flat non-unitary line bundles as in \cite{Cho}. Then the potential function is expressed as
\begin{equation}\label{certainpotential}
W \left(L ; b \right) = \sum_{\beta} n_\beta \cdot \exp(\pa \beta \cap b) \, T^{\omega(\beta) / 2 \pi}
\end{equation}
where the summation is taken over all homotopy classes in $\pi_2 (X, L)$ of Maslov index two.

Setting
\begin{equation}\label{gamman}
\Gamma(n) = \{ (i, j) \mid 2 \leq i + j \leq n\},
\end{equation}
we fix the basis $\{ \gamma_{i,j} \mid (i,j) \in \Gamma(n) \}$ for $H^1(L; \Z)$ dual to the basis
for $H_1(L, \Z)$ consisting of the orbits generated by periodic Hamiltonians $\{ u_{i,j} \mid (i,j) \in \Gamma(n) \}$ in \cite{GSflag}. 
A $1$-cochain $b \in H^1(L; \Lambda_0)$ is expressed as the linear combination  $\sum_{(i,j) \in \Gamma(n)} x_{i,j} \cdot \gamma_{i,j}$ and  
take the exponential variables
\begin{equation}\label{exponencoord}y_{i,j} := e^{x_{i,j}}.
\end{equation}
Then the potential function can be expressed as a Laurent polynomial $W(\mathbf{y})$ with respect to $\{ y_{i,j} \mid (i,j) \in \Gamma(n) \}$. Setting $u_{i,n+1-i}:= \lambda_{i}$, keep in mind that $\Delta_\lambda$ is defined by
$$
\left\{ (u_{i,j}) \in \R^{n(n-1)/2} \mid u_{i, j+1} - u_{i,j} \geq 0, \, u_{i,j} - u_{i+1, j} \geq 0 \, \right\}.
$$

\begin{theorem}[Theorem 10.1 in \cite{NNU}]
Consider the Lagrangian torus $L$ over a point $( u_{i,j} \mid (i,j) \in \Gamma(n) )$
and set $u_{i,n+1-i}= :\lambda_{i}$.
Then the potential function on $L$ is given by
\begin{equation}\label{potentialfunctionoriginal}
W(L; \textup{\textbf{y}}) = \sum_{(i,j)} \left( \frac{y_{i, j+1}}{y_{i,j}}
\, T^{ u_{i,j+1} - u_{i,j}} + \frac{y_{i,j}}{y_{i+1, j}} \, T^{ u_{i,j} - u_{i+1, j}}  \right).
\end{equation}
\end{theorem}

For $t$ with $0 \leq t < 1$, we can arrange the potential function of the fibers $L_m(t)$ over $I_m(t)$ as
\begin{align}\label{potential}
\begin{split}
&W(L_m(t); \mathbf{y}) = \left( \sum_{i = 1}^{m} \sum_{j=1}^{m-1} \frac{y_{i, j+1}}{y_{i,j}}+ \sum_{i= 1}^{m-1} \sum_{j=1}^{m} \frac{y_{i, j}}{y_{i+1,j}}\right) T^{1-t} \\ 
& \,\,\,\,\, + \left(  \sum_{\substack{\max(i, j) \geq m+1 }} \left( \frac{y_{i, j+1}}{y_{i,j}}+ \frac{y_{i, j}}{y_{i+1,j}} \right) \right) T^{1} 
+ \left( \sum_{i=0}^{m-1} \left( \frac{y_{m-i,m+1}}{y_{m-i,m}} + \frac{y_{m,m-i}}{y_{m+1,m-i}} \right) \right)  T^{1 + it} 
\end{split}
\end{align}
For simplicity, we frequently omit $L_m(t)$ in $W(L_m(t); \mathbf{y})$ if $L_m(t)$ is clear in the context.

\subsection{Bulk-deformations by Schubert cycles}\label{section:bulkdeform}

We shall apply Lagrangian Floer theory deformed by ambient cycles of a symplectic manifold, developed
 in \cite{FOOO, FOOOToric2}. We will exploit the cycles of the form
$$
\frak{b}:= \sum_{j=1}^B \frak{b}_j \cdot \scr{D}_j,
$$
where each $\scr{D}_j$ is a cycle of degree two not intersecting $L$.
The deformed potential function is denoted by $W^\frak{b}$.

We first recall  from \cite{FOOOToric2} the formula for the potential function of a torus fiber $L$ deformed by a combination of toric divisors $\scr{D}_j$,
$
\frak{b}:= \sum_{j=1}^B \frak{b}_j \cdot \scr{D}_j,
$
in a compact toric manifold $X$.

\begin{theorem}[\cite{FOOOToric2}]\label{Foootoric2potenti} The bulk-deformed potential function, also called the potential function with bulk, is \begin{equation}\label{bulkdeformedpotential}
W^\frak{b}\left( L ; b \right) = \sum_{\beta} n_\beta \cdot \exp \left( \sum_{j=1}^B \left( \beta \cap \scr{D}_{j} \right) \frak{b}_{j} \right) \exp(\pa \beta \cap b) \, T^{\omega(\beta) / 2 \pi}.
\end{equation}
where the summation is taken over all homotopy classes in $\pi_2 (X, L)$ of Maslov index two.
\end{theorem}

In the derivation of this in \cite{FOOOToric2}, the smoothness and the $T^n$-invariance of
the relevant ambient cycles are used. Since Schubert cycles that we will use 
are neither smooth nor $T^n$-invariant in general, we will provide details of
the proof of this theorem for the current GC case modifying the arguments used in the proof of \cite[Proposition 4.7]{FOOOToric2} similarly as done in \cite[Section 9]{NNU}, see Section~\ref{sec_bulkdefbyschucy}. The upshot is that we still have the same formula for the potential function with bulk in the current GC case, see~\eqref{bulkdeformedpotential}.
Again by taking the system of exponential coordinates in~\eqref{exponencoord}, $W^\frak{b}$ in~\eqref{bulkdeformedpotential} becomes a Laurent polynomial with respect to $\{ y_{i,j} \mid (i,j) \in \Gamma(n) \}$.

\begin{theorem}[Section 8 in \cite{FOOOToric2}]\label{foootoric2nondispl}\label{criticalpointimpliesnondisplaceability}
If the bulk-deformed potential function $W^\frak{b} (L; \mathbf{y})$ admits a critical point $\mathbf{y}$ whose components are in $\Lambda_U$, then $L$ is non-displaceable.
\end{theorem}

In his thesis \cite{Ko}, Kogan found an expression of a Schubert cycle in terms of a certain union of the inverse images of faces in the GC system of a complete flag manifold, see also Kogan--Miller \cite{KoM}. Due to presence of non-torus fibers in \cite{CKO}, the inverse image of a \emph{single} face might have boundary so that it does \emph{not} form a cycle. What he proved is that a certain combination of faces can form a cycle because the boundaries are cancelled out.

We review the result in terms of ladder diagrams. A facet in a GC polytope is called \emph{horizontal} (resp. \emph{vertical}) if it is given by $u_{i,j} = u_{i+1,j}$ (resp. $u_{i, j+1} = u_{i,j}$). Let $P^\textup{hor}_{i,i+1}$ (resp. $P^\textup{ver}_{j+1,j}$) be the union of horizontal (resp. vertical) facets between the $i$-th column and the $(i+1)$-th column (resp. the $(j+1)$-th row and the $j$-th row). That is,
\begin{align*}
P^\textup{hor}_{i,i+1} := \bigcup_{s=1}^{n-i} \{ u_{i, s} = u_{i+1, s} \}, \quad P^\textup{ver}_{j+1, j} := \bigcup_{r=1}^{n-j} \{ u_{r, j+1} = u_{r, j} \}
\end{align*}
for $1 \leq i, j  \leq n - 1$ where $\{ u_{\bullet, \bullet} = u_{\bullet, \bullet}\}$ denotes the facet given by the equation inside.
Let
\begin{equation}\label{eq:D-ver-hor}
\scr{D}^\textup{hor}_{i,i+1} := \Phi^{-1}_\lambda \left( P^\textup{hor}_{i,i+1} \right), \quad
\scr{D}^\textup{ver}_{j+1,j} := \Phi^{-1}_\lambda \left( P^\textup{ver}_{j+1,j} \right),
\end{equation}
which are respectively called a \emph{horizontal} and \emph{vertical Schubert cycle} (of degree two).
(See Theorem~\ref{thm_KoganMiller} below.)

\begin{example}
Consider the co-adjoint orbit $\mcal{O}_\lambda \simeq \mathrm{Fl}(6)$ where $\lambda = (5, 3, 1, -1, -3, -5)$.
 $P^\textup{hor}_{4,5}$ is the union of two horizontal facets $P^\textup{hor}_{4,5} = \{ u_{4,2} = -3 \} \cup \{ u_{4,1} = u_{5,1} \}$
as in Figure~\ref{HorizontalP45} and $P^\textup{ver}_{4,3}$ is the union of three vertical facets $P^\textup{ver}_{4,3} = \{ 1 = u_{3,3} \} \cup \{ u_{2,4} = u_{2,3} \} \cup \{u_{1,4} = u_{1,3} \}$
as in Figure~\ref{HorizontalP45}.

\begin{figure}[ht]
	\scalebox{0.9}{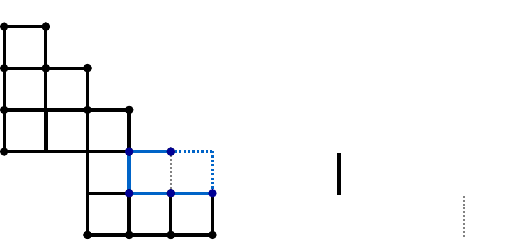}
	\quad \quad \quad \quad \,\,\,
		\scalebox{0.9}{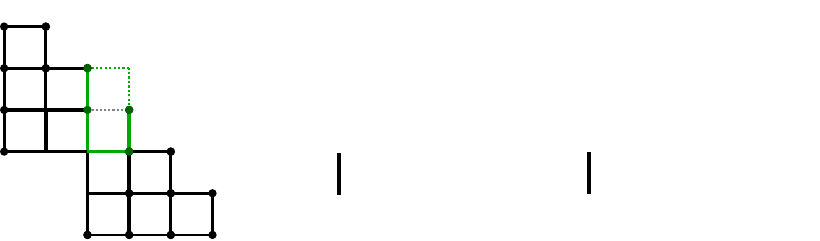}
	\caption{\label{HorizontalP45} $P^\textup{hor}_{4,5}$ and $P^\textup{ver}_{4,3}$ in $\mathrm{Fl}(6)$}	
\end{figure}
\end{example}

From the combinatorial process in \cite{Ko,KoM}, we observe that the Schubert varieties associated with the simple transpositions
with a complex codimension one are corresponding to either unions of horizontal facets or unions of vertical facets. The opposite Schubert varieties are corresponding to the other \cite[Remark 9]{KoM}.

\begin{theorem}[Theorem 2.3.1 in \cite{Ko}, Theorem 8 in \cite{KoM}]\label{thm_KoganMiller}
The inverse image $\scr{D}^\textup{hor}_{\bullet, \bullet+1}$ (or $\scr{D}^\textup{ver}_{\bullet+1, \bullet}$) represents an (or opposite) Schubert cycle of degree two.
\end{theorem}

Now, we apply~\eqref{bulkdeformedpotential}, which is a counterpart of~\eqref{Foootoric2potenti}, to calculate the bulk-deformed potential function. Using the one-to-one correspondence in  the property (3) in Section~\ref{reviewpotentialpotentialfunc}, let $\beta^{i,j}_{i+1, j} (\textup{resp. } \beta^{i,j+1}_{i, j} )$
be the homotopy class in $\pi_2(\mcal{O}_\lambda, L)$ represented by a holomorphic disc intersecting the facet $u_{i,j} = u_{i+1,j}$ (resp. $u_{i,j+1} = u_{i,j}$) once.

\begin{lemma}\label{intersectionnumbercal}
Let $\scr{D}$ be either a horizontal or a vertical Schubert cycle in $X_s$. Then we have
$$
\beta^{i,j}_{i+1,j} \cap \scr{D} =
\begin{cases}
1 \quad \mbox{if } \scr{D} = \scr{D}^\textup{hor}_{i, i+1} \\
0 \quad \mbox{otherwise, }
\end{cases}
\quad
\beta^{i,j+1}_{i,j} \cap \scr{D} =
\begin{cases}
1 \quad \mbox{if } \scr{D} = \scr{D}^\textup{ver}_{j+1, j} \\
0 \quad \mbox{otherwise. }
\end{cases}
$$
\end{lemma}

\begin{proof}
Let $\phi^\prime_s \colon X_s \to X_0$ be a (continuous) extension of the flow $\phi_s \colon X^{\textup{sm}}_s \to X^{\textup{sm}}_0$
in Theorem~\ref{NNUtoricdeg}, see \cite[Section 8]{NNU}.
Let $\varphi \colon (\mathbb{D}, \pa \mathbb{D}) \to (X_s, L_s)$ be a holomorphic disc in the class $\beta^{i,j}_{i+1,j}$ of Maslov index two for example.
We then have a (topological) disc $\phi^\prime_s \circ \varphi \colon (\mathbb{D}, \pa \mathbb{D}) \to (X_s, L_s) \to (X_0, L_0)$, representing $(\phi^\prime_s)_* \beta^{i,j}_{i+1,j}$. Note that there exists a holomorphic disc $\varphi_0$ by \cite{CO} in the class $(\phi^\prime_s)_* \beta^{i,j}_{i+1,j} = [\phi^\prime_s \circ \varphi]$. Meanwhile, by our choice of $\scr{D}$, $\phi^\prime_s (\scr{D})$ is the union of the components over either $P^\textup{hor}_{i, i+1}$ or $P^\textup{ver}_{j+1, j}$.
Since the flow $\phi^\prime_s$ gives rise to a symplectomorphism from $X^{\textup{sm}}_s$ to $X^{\textup{sm}}_0$ and the image of the disc $\varphi$ is contained in $X^{\textup{sm}}_s$, the (local) intersection number should be preserved through the flow $\phi^\prime_s$. To calculate the intersection number, we consider a small resolution $p \colon \widetilde{X}_0 \to X_0$. Because the intersection happens only outside the singular loci of $X_0$, we can lift the divisor and the disc $\varphi_0$ to $\widetilde{\scr{D}}_0$ and $\widetilde{\varphi}_0$ in $\widetilde{X}_0$ without any change of the intersection number. Then we have
$$
\beta^{i,j}_{i+1,j} \cap \scr{D} = [ \varphi] \cap \scr{D} = [ \widetilde{\varphi}_0 ]\cap \widetilde{\scr{D}}_0,
$$
which completes the proof.
\end{proof}

We take
\begin{equation}\label{bulkparak}
\frak{b} :=  \sum_{i} \frak{b}^\textup{hor}_{i, i+1} \cdot \varphi^\prime_{1-s} \left( \scr{D}^\textup{hor}_{i, i+1}  \right)
 + \sum_{j} \frak{b}^\textup{ver}_{j+1, j} \cdot \varphi^\prime_{1-s} \left( \scr{D}^\textup{ver}_{j+1, j} \right)
\end{equation}
where $\frak{b}^\textup{hor}_{i, i+1}, \frak{b}^\textup{ver}_{j+1, j} \in \Lambda_0$ and $\varphi^\prime_{1-s} \colon X_1 \to X_{s}$. By abuse of notation for simplicity, we denote $\varphi^\prime_{1-s} \left( \scr{D}^\textup{hor}_{i, i+1}  \right)$ (resp. $\varphi^\prime_{1-s} \left( \scr{D}^\textup{ver}_{j+1, j} \right)$) by $\scr{D}^\textup{hor}_{i, i+1}$ (resp. $\scr{D}^\textup{ver}_{j+1, j}$). By the homotopy invariance of the $A_\infty$-structures, we calculate the (bulk-deformed) Floer cohomology of $L_{m,s}(t)$ in $X_s$ for $s$ sufficiently close to $0$. In particular, non-displaceability of $L_{m}(t)$ can be achieved as long as the Floer cohomology of $L_{m,s}(t)$ is non-zero. Whenever turning on a bulk-deformation, this process passing to $L_{m,s}(t)$ and $\varphi^\prime_{1-s} (\scr{D})$  in $X_s$ will be taken into consideration.  Also, depending on the position $t$ of a Lagrangian torus $L_m( \, \cdot \,)$, we need to consider different $\frak{b}^\textup{hor}_{i, i+1}$ and $\frak{b}^\textup{ver}_{j+1, j}$.

\begin{corollary}\label{formulaforbulkdeformedpotentialoursitu}
Taking a bulk-deformation parameter as in~\eqref{bulkparak}, the deformed potential function is expressed as
\begin{equation}
W^\frak{b}(L; \textup{\textbf{y}}) = \sum_{(i,j)} \left( \exp \left( \frak{b}^\textup{hor}_{i,i+1} \right) \frac{y_{i,j}}{y_{i+1, j}}  T^{ u_{i,j} - u_{i+1, j}} + \exp \left( \frak{b}^\textup{ver}_{j+1,j} \right) \frac{y_{i, j+1}}{y_{i,j}}  T^{ u_{i,j+1} - u_{i,j}} \right).
\end{equation}
\end{corollary}
Let
$$
 c^\textup{hor}_{i,i+1} := \exp (\frak{b}^\textup{hor}_{i, i+1}), \quad
 c^\textup{ver}_{j+1,j} := \exp (\frak{b}^\textup{ver}_{j+1, j}).
$$
By definition, they lie in $\Lambda_U$. With this notation,
the logarithmic derivative of the bulk-deformed potential becomes
\begin{equation}\label{thegradientofbulkdeformedpotential}
\begin{split}
y_{i,j} \frac{\pa W^{\frak{b}}}{\pa y_{i,j}}({\mathbf{y}}) &= - c^\textup{ver}_{j+1, j}  \frac{y_{i, j+1}}{y_{i, j}} T^{u_{i,j+1} - u_{i,j}} - c^\textup{hor}_{i-1, i} \frac{y_{i-1, j}}{y_{i,j}} T^{u_{i-1,j} - u_{i,j}}\\
&+ c^\textup{hor}_{i, i+1} \frac{y_{i,j}}{y_{i+1, j}} T^{u_{i,j} - u_{i+1,j}} + c^\textup{ver}_{j, j-1} \cdot \frac{y_{i, j}}{y_{i, j-1}} T^{u_{i,j} - u_{i,j-1}}
\end{split}
\end{equation}

\subsection{Non-displaceable Gelfand--Cetlin fibers in $\mathrm{Fl}(3)$}\label{pfofmaincorforfullflag3}

In this section, the case of $\mathrm{Fl}(3)$ will be discussed in details. The following theorem will be proven
as a warmup towards the general theorem.

\begin{theorem}[Theorem~\ref{theoremC}]\label{theorem_maincompleteflag3}
Let $\lambda = ( \lambda_1 = 2 > \lambda_2 = 0 > \lambda_3 = -2 )$.
Consider the co-adjoint orbit $\mathcal{O}_\lambda$, a complete flag manifold $\mathrm{Fl}(3)$ equipped with the monotone Kirillov--Kostant--Souriau symplectic form $\omega_\lambda$, Then the Gelfand--Cetlin fiber over a point $\mathbf{u} \in \Delta_\lambda$ is non-displaceable if and only if $\mathbf{{u}} \in I$ where
\begin{equation}\label{linesegmentforf3}
I := \left\{ (u_{1,1}, u_{1,2}, u_{2,1}) = (0, 1 -t, -1 +t) \in \R^3 ~|~ 0 \leq t \leq 1 \right\}
\end{equation}
In particular, the Lagrangian 3-sphere $\Phi_\lambda^{-1}(0,0,0)$ is non-displaceable.
\end{theorem}

To deform Floer theory, we employ a combination of the vertical and horizontal divisors in~\eqref{eq:D-ver-hor}
any of which do \emph{not} intersect the torus fibers. Let
\begin{equation}\label{eq_combverhor}
\frak{b} = \frak{b}^\textup{ver}_{2,1} \cdot \scr{D}^\textup{ver}_{2,1} + \frak{b}^\textup{hor}_{1,2} \cdot \scr{D}^\textup{hor}_{1,2} + \frak{b}^\textup{ver}_{3,2} \cdot \scr{D}^\textup{ver}_{3,2} + \frak{b}^\textup{hor}_{2,3} \cdot \scr{D}^\textup{hor}_{2,3}.
\end{equation}
For the proof of Theorem~\ref{theorem_maincompleteflag3}, we need the following topological fact.

\begin{proposition}\label{closednessofnondisp}
Let $\Phi \colon X \to \Delta \subset \R^d$ be a completely integrable system such that $\Phi$ is proper. If there exists a sequence $\{ \mathbf{u}_i \mid i \in \N \}$ such that
\begin{enumerate}
\item Each $\Phi^{-1}(\mathbf{u}_i)$ is non-displaceable.
\item The sequence $\mathbf{u}_i$ converges to some point $\mathbf{u}_\infty$ in $\Delta$.
\end{enumerate}
then $\Phi^{-1}(\mathbf{u}_\infty)$ is also non-displaceable.
\end{proposition}

\begin{proof}
For a contradiction, suppose that $\Phi^{-1}(\mathbf{u}_\infty)$ is displaceable. There is a Hamiltonian diffeomorphism $\phi$ and an open set $U$ containing $\Phi^{-1}(\mathbf{u}_\infty)$ in $X$ such that $\phi(U) \cap U = \emptyset$.
For each $i$, there exists a point $x_i \in \Phi^{-1}(\mathbf{u}_i)$ such that $x_i \notin U$ since $\Phi^{-1}(\mathbf{u}_i)$ is non-displaceable. It implies that any subsequence of $\{x_i\}$ cannot converge to a point in $U$. On the other hand, passing to a subsequence, we may assume that $x_i$ converges to $x_\infty$ for some $x_\infty \in X$ since $\Phi$ is proper. By the continuity of $\Phi$, we then have
$$
\mathbf{u}_\infty = \lim_{i \to \infty} \mathbf{u}_i = \lim_{i \to \infty} \Phi (x_i) = \Phi (x_\infty).
$$
It leads to a contradiction that $x_\infty \in \Phi^{-1} (\mathbf{u}_\infty) \subset U$.
\end{proof}

We now start the proof of Theorem~\ref{theorem_maincompleteflag3}.

\begin{proof}[Proof of Theorem~\ref{theorem_maincompleteflag3}]

For any fixed $t$ with $0 \leq t < 1$, let $L(t)$ be the Lagrangian torus fiber over $(0, 1-t, -1+t) \in I$ in~\eqref{linesegmentforf3}. Let $L_s(t)$ be the fiber corresponding to $L(t)$ in $X_s$ via a toric degeneration of completely integrable systems in Theorem~\ref{NNUtoricdeg}. By taking $s > 0$ sufficiently close to $0$, the potential function of $L_s(t)$ can be arranged as
$$
W(\mathbf{y}) = \left( \frac{y_{1,2}}{y_{1,1}} + \frac{y_{1,1}}{y_{2,1}} + y_{1,2} + \frac{1}{y_{2,1}} \right) T^{1-t} + \left( \frac{1}{y_{1,2}} + y_{2,1} \right) T^{1+t}.
$$
We use a combination of Schubert cycles in~\eqref{bulkparak} (or equivalently~\eqref{eq_combverhor}) to deform the potential function. A strategy we take is to postpone determining bulk-deformation parameters. Namely, we start with a tentative parameter, determine solutions for $\mathbf{y}$ first, and then adjust the parameter to make the chosen $\mathbf{y}$ a critical point.

Take a \emph{tentative} bulk-parameter $\frak{b}^\prime := \frak{b}^\textup{ver}_{2,1} \cdot \scr{D}^\textup{ver}_{2,1}$ such that $\exp (\frak{b}^\textup{ver}_{2,1}) = 1 + T^{2t}$, i.e.,
$$
\frak{b}^\textup{ver}_{2,1} = T^{2t} - \frac{1}{2} T^{4t} + \cdots \in \Lambda_+.
$$
By Corollary~\ref{formulaforbulkdeformedpotentialoursitu}, the potential function is deformed into
$$
W^{\frak{b}^\prime}(\mathbf{y}) = \left( \frac{y_{1,2}}{y_{1,1}} + \frac{y_{1,1}}{y_{2,1}} + y_{1,2} + \frac{1}{y_{2,1}} \right) T^{1-t} + \left( \frac{y_{1,2}}{y_{1,1}} + \frac{1}{y_{2,1}} + \frac{1}{y_{1,2}} + y_{2,1} \right) T^{1+t},
$$
whose logarithmic derivatives are
$$
\begin{cases}
\displaystyle y_{1,1} \frac{\pa W^{\frak{b}^\prime}}{\pa y_{1,1}} (\mathbf{y}) =  \left( - \frac{y_{1,2}}{y_{1,1}} + \frac{y_{1,1}}{y_{2,1}} \right) T^{1-t} + \left( - \frac{y_{1,2}}{y_{1,1}} \right) T^{1+t} \\
\displaystyle y_{1,2} \frac{\pa W^{\frak{b}^\prime}}{\pa y_{1,2}} (\mathbf{y}) =  \left(  \frac{y_{1,2}}{y_{1,1}} + {y_{1,2}} \right) T^{1-t} + \left( \frac{y_{1,2}}{y_{1,1}} - \frac{1}{y_{1,2}} \right) T^{1+t} \\
\displaystyle y_{2,1} \frac{\pa W^{\frak{b}^\prime}}{\pa y_{2,1}} (\mathbf{y}) =  \left( - \frac{y_{1,1}}{y_{2,1}} - \frac{1}{y_{2,1}} \right) T^{1-t} + \left( - \frac{1}{y_{2,1}} + y_{2,1} \right) T^{1+t}.
\end{cases}
$$

We set $y_{1,2} = 1, y_{2,1} = 1$ and take $y_{1,1}$ as the solution of $(y_{1,1})^2 = 1 + T^{2t}$ satisfying $y_{1,1} \equiv -1 \mod T^{>0}$. It is easy to see that $y_{1,1} \frac{\pa W^{\frak{b}^\prime}}{\pa y_{1,1}} (\mathbf{y}) = 0$. Note that $y_{1,1}$ is of the form
$
y_{1,1} \equiv - 1 - \frac{1}{2} T^{2t} \mod T^{>2t}.
$

We now adjust a bulk-deformation parameter from $\frak{b}^\prime$ to $\frak{b}$ in order for the chosen $(y_{1,1}, y_{1,2}, y_{2,1})$ to be a critical point of $W^\frak{b}$. Let
$$
\frak{b} := \frak{b}^\prime + \frak{b}^\textup{ver}_{3,2} \cdot \scr{D}^\textup{ver}_{3,2} + \frak{b}^\textup{hor}_{2,3} \cdot \scr{D}^\textup{hor}_{2,3}.
$$
Since $\scr{D}^\textup{ver}_{3,2}$ and $\scr{D}^\textup{hor}_{2,3}$ do not intersect the discs of Maslov index $2$ emanating from $\scr{D}^\textup{ver}_{2,1}$ and $\scr{D}^\textup{hor}_{1,2}$  in $\pi_2(\mcal{O}_\lambda, \Phi^{-1}_\lambda(t))$, we have
$$
y_{1,1} \frac{\pa W^{\frak{b}}}{\pa y_{1,1}} (\mathbf{y}) = y_{1,1} \frac{\pa W^{\frak{b}^\prime}}{\pa y_{1,1}} (\mathbf{y}).
$$
Plugging the chosen $y_{i,j}$'s, we have
$$
\begin{cases}
\displaystyle y_{1,2} \frac{\pa W^{\frak{b}}}{\pa y_{1,2}} (\mathbf{y}) = \left( -\frac{1}{2} - \exp (\frak{b}^\textup{ver}_{3,2}) \right) T^{1+t} + h^{(1,2)} \cdot T^{1+t}  \\
\displaystyle y_{2,1} \frac{\pa W^{\frak{b}}}{\pa y_{2,1}} (\mathbf{y}) = \left( -\frac{1}{2} + \exp (\frak{b}^\textup{hor}_{2,3}) \right) T^{1+t} + h^{(2,1)} \cdot T^{1+t}.
\end{cases}
$$
for some constant $h^{(1,2)}, h^{(2,1)}\in \Lambda_+$. By choosing $\frak{b}^\textup{ver}_{3,2}, \frak{b}^\textup{hor}_{2,3} \in \Lambda_0$ such  that $\exp (\frak{b}^\textup{ver}_{3,2}) = -\frac{1}{2} + h^{(1,2)}$ and $\exp (\frak{b}^\textup{hor}_{2,3}) = \frac{1}{2} - h^{(2,1)}$,
we can make $W^\frak{b}(\mathbf{y})$ admit a critical point. By Theorem~\ref{criticalpointimpliesnondisplaceability}, $L_s(t)$ has a non-vanishing (bulk-)deformed Floer cohomology. By the Hamiltonian invariance of $A_\infty$-structures, so does $L(t)$ and therefore it is non-displaceable.

In sum, each torus fiber over the line segment in~\eqref{linesegmentforf3} is non-displaceable. Furthermore, Proposition~\ref{closednessofnondisp} yields non-displaceability of the Lagrangain $3$-sphere.
\end{proof}

\section{Decompositions of the gradient of potential function}
\label{secDecompositionsOfTheGradientOfPotentialFunction}
\label{section:decompositionofpote}

In this section, in order to prove Theorem~\ref{theoremD}, we introduce \emph{the split leading term equation} of the potential function in~\eqref{potential}, which is the analogue of \emph{the leading term equation} in~\cite{FOOOToric1, FOOOToric2}. We discuss the relation between its solvability and non-triviality of Floer cohomology under a certain choice of a bulk-deformation.

\subsection{Outline of Section~\ref{section:decompositionofpote} and~\ref{solvabilityofthesplitleadingtermequ}}

Thanks to Theorem~\ref{criticalpointimpliesnondisplaceability}, the proof of
Theorem \ref{theoremD}
boils down to finding a bulk-deformation parameter $\frak{b}$ such that the bulk-deformed
potential function $W^\frak{b}$ admits a critical point. Section~\ref{section:decompositionofpote}
and~\ref{solvabilityofthesplitleadingtermequ} will be occupied to discuss how to determine them.

Before giving the outline, we start with explaining why this process is non-trivial by pointing out
the differences from the toric case. In the toric case, the (generalized) leading term equation was
introduced to detect non-displaceable toric fibers effectively in~\cite[Section 11]{FOOOToric2}.
Roughly speaking, it consists of the initial terms of the gradient of a (bulk-deformed) potential
function with respect to a suitable choice of exponential variables. It is proven therein that there
always exists a bulk parameter $\frak{b}$ so that the complex solution becomes a critical point of
$W^\frak{b}$ if the leading term equation admits a solution whose components are non-zero. Indeed,
the locations where the leading term equation is solvable are characterized by the intersection of
 certain tropicalizations in \cite{KLS}. The key features for proving the above statements are as follows.
 First, there is a one-to-one correspondence between the \emph{honest} holomorphic discs bounded by a
 torus fiber of Maslov index $2$ and the facets of the moment polytope. Second, the pre-image of each
 facet represents a cycle of degree $2$. Thus, all terms corresponding to the facets can be independently controlled.

In the GC system, the inverse image of a single facet is \emph{not} a cycle in general because of
the appearance of non-torus fibers at the boundary.
We need to take into account a particular union of facets to represent a cycle of degree $2$.
The terms of $W$ \emph{cannot} be independently controlled if taking such a cycle for a bulk-deformation.
Nonetheless, for the family of Lagrangian tori $L_m(t)$ over~\eqref{IMT} in $\mcal{O}_\lambda \simeq \mathrm{Fl}(n)$
with the monotone symplectic form $\omega_\lambda$, we shall show the existence of a
bulk-deformation parameter $\frak{b}$ and a critical point associated to $\frak{b}$.

In Section~\ref{section:decompositionofpote}, we introduce the \emph{split leading term equation}, see Definition~\ref{splittingleadingtermequa}. We then demonstrate how to determine a bulk-deformation
parameter and promote a solution of the split leading term equation to a critical point of the bulk-deformed
potential function once a solution of the split leading term equation is given.
In Section~\ref{solvabilityofthesplitleadingtermequ}, we show that the split leading term
equation always admits a solution. 
One might think the situation looks somewhat similar to that of Bernstein-Kushinirenko
\cite{Be,Ku} which, however, treats the \emph{generic} system of Laurent polynomial equations while
the system in our situation does not have much freedom. For a given fixed system of multi-variable equations
finding a solution is not simple at all even with the aid of a computer.
Yet, guided by the ladder diagrams decorated by the exponential variables and using the freedom
in the choice of \emph{seeds} (see Section \ref{subsec:seeds}) which provides some freedom of changing
the coefficients of the equations,  we are able to show the existence of a solution for the system
of our current interest.

A brief description of our procedure of solving the critical point equation is now in order.
Let $m$ be any fixed integer with $2 \leq m \leq \lfloor n/2 \rfloor$ and
$B(m)$ be the sub-diagram consisting of $(m \times m)$ lower-left unit boxes in the ladder
diagram $\Gamma(n)$ of $\mathrm{Fl}(n)$. The diagrams $\Gamma(n)$ and $B(m)$ are often regarded
as collections of double indices as follows:
\begin{equation}\label{gammanbm}
\begin{split}
&\Gamma(n) = \{ (i, j) \mid 2 \leq i + j \leq n\}\\
&B(m) = \{ (i, j) \mid 1 \leq i, j \leq m\}.
\end{split}
\end{equation}
Recalling~\eqref{potential}, the potential function of $L_m(t)$ is rearranged into several groups
in terms of the energy filtration. The valuation of $\pa_{(i,j)}  W(\mathbf{y})$ for $(i, j) \in B(m)$ is $(1-t)$ and that of $\pa_{(i,j)}  W(\mathbf{y})$ for $(i, j) \in \Gamma(n) \backslash B(m)$ is $1$.

Decomposing the gradient of the potential function deformed by $\frak{b}$ in~\eqref{bulkparak} into two pieces along the boundary of $B(m)$, we determine a critical point in the following steps.
\begin{enumerate}
\item Find a solution $(y^\C_{i,j} \in \C^*)$ of the system of the equations $\pa_{(i,j)}  W^\frak{b}(\mathbf{y}) \equiv 0 \mod T^{>1}$ in $(\Gamma(n) \backslash  B(m)) \cup \{(m,m)\}$ and equations relating the variables adjacent to $B(m)$ in Section~\ref{solvabilityofthesplitleadingtermequ}.
\item Find a solution $(y^\C_{i,j} \in \C^*)$ of $\pa_{(i,j)}  W^\frak{b}(\mathbf{y}) \equiv 0 \mod T^{>1-t}$ in $B(m)$ in Section~\ref{symmetriccomplexsolinbm}.
\item Find a solution $(y_{i,j} \in \Lambda_U)$ of $\pa_{(i,j)}  W^\frak{b}(\mathbf{y}) = 0$ in $B(m)$  such that $y_{i,j} \equiv y_{i,j}^\C \mod T^{>0}$ in Section~\ref{insidebm}.
\item Find a solution $(y_{i,j} \in \Lambda_U)$ of $\pa_{(i,j)}  W^\frak{b}(\mathbf{y}) = 0$ in $\Gamma(n) \backslash  B(m)$ such that $y_{i,j} \equiv y_{i,j}^\C \mod T^{>0}$ in Section~\ref{outsideofbm}.
\end{enumerate}
The \emph{split leading term equation} (See Definition~\ref{splittingleadingtermequa}) arises in the first step $(1)$.

\begin{example}
In the co-adjoint orbit $\mcal{O}_\lambda$ where $\lambda = ( 5, 3, 1, -1, -3, -5 )$, the potential function of $L_2(t)$ is
$$
W(L_2(t); \mathbf{y}) = \left( \frac{y_{1, 2}}{y_{1,1}}+ \frac{y_{1, 1}}{y_{2,1}} + \frac{y_{1, 2}}{y_{2,2}} + \frac{y_{2, 2}}{y_{2,1}} \right) T^{1-t} + \left(  \frac{y_{1, 4}}{y_{1,3}} + \frac{y_{1, 3}}{y_{2,3}} + \cdots \right) T^{1} + \left( \frac{y_{1, 3}}{y_{1,2}} + \frac{y_{2, 1}}{y_{3,1}} \right) T^{1 + t}.
$$
In this example, the valuation of partial derivatives of $W$ jumps along the red line in Figure~\ref{Decomposition of the gradient of the potential function}.

Turning on bulk-deformation, it follows from ~\eqref{thegradientofbulkdeformedpotential} that
any solution complex number $y_{i,j} \in \C^*$ of the split leading term equation has to satisfy
the system
\begin{equation}\label{firstsysfrominb2}
\begin{cases}  - c^\textup{ver}_{2,1} \cdot \frac{y_{1, 2}}{y_{1,1}}+ c^\textup{hor}_{1,2} \cdot \frac{y_{1, 1}}{y_{2,1}} = 0, \,\, c^\textup{ver}_{2,1} \cdot \frac{y_{1, 2}}{y_{1,1}} + c^\textup{hor}_{1,2} \cdot \frac{y_{1, 2}}{y_{2,2}} = 0, \\
 - c^\textup{hor}_{1,2} \cdot \frac{y_{1, 1}}{y_{2,1}} - c^\textup{ver}_{2,1} \cdot \frac{y_{2, 2}}{y_{2,1}} = 0, \,\, - c^\textup{hor}_{1,2} \cdot \frac{y_{1, 2}}{y_{2,2}} + c^\textup{ver}_{2,1} \cdot \frac{y_{2, 2}}{y_{2,1}} = 0,
\end{cases}
\end{equation}
and
\begin{equation}\label{secondsysfromoutb22}
\begin{cases}
- c_{4 - i + 1, 4-i}^\textup{ver} \frac{y_{i,4 - i + 1}}{y_{i,4 - i}} + c_{i,i+1}^\textup{hor}  \frac{y_{i, 5 - i}}{y_{i + 1, 5- i}} = 0 &\mbox{ for $i = 1, \dots, 3$.} \\
- c_{5 - i + 1, 5-i}^\textup{ver}  \frac{y_{i,5 - i + 1}}{y_{i,5 - i}} - c_{i-1, i}^\textup{hor}  \frac{y_{i-1,5-i}}{y_{i,5-i}} 
+ c_{5 - i, 5-i-1}^\textup{ver}  \frac{y_{i,5 - i}}{y_{i,5 - i - 1}}
+ c_{i,i+1}^\textup{hor} \cdot \frac{y_{i, 4 - i}}{y_{i + 1, 4- i}} = 0 &\mbox{ for $i = 1, \dots, 4$.} \\
- c_{6 - i + 1, 6-i}^\textup{ver}  \frac{1}{y_{i,6 - i}} - c_{i-1, i}^\textup{hor}  \frac{y_{i-1,6-i}}{y_{i,6-i}} 
+ c_{6 - i, 6-i-1}^\textup{ver}  \frac{y_{i,6 - i}}{y_{i,6 - i - 1}}
+ c_{i,i+1}^\textup{hor} {y_{i, 5 - i}} = 0 &\mbox{ for $i = 1, \dots, 5$.} 
\end{cases}
\end{equation}
The equations~\eqref{firstsysfrominb2} and~\eqref{secondsysfromoutb22} come from the leading parts of the partial derivatives inside $B(2)$ and of the partial derivatives inside $\Gamma(6) \backslash B(2) \cup \{(2,2)\}$, respectively. Solving ~\eqref{firstsysfrominb2} corresponds to the step $(2)$ and solving ~\eqref{secondsysfromoutb22} corresponds to the step $(1)$. In the rest of this section, assuming that the split leading term equation is solvable,
we explain how to complete the remaining steps.

\begin{figure}[ht]
	\scalebox{0.85}{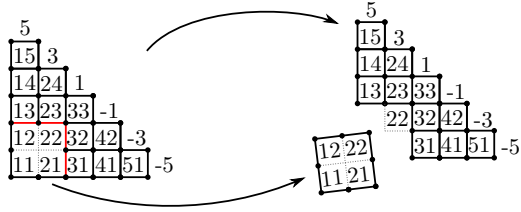}
	\caption{\label{Decomposition of the gradient of the potential function} Decomposition of the gradient of the potential function in $\mathrm{Fl}(6)$.}	
\end{figure}
\end{example}

\subsection{Split leading term equations}

Let $\lambda = ( \lambda_{i} := n - 2i + 1 \mid i = 1, \cdots, n )$ be an $n$-tuple of real numbers for an arbitrary integer $n \geq 4$. Consider the co-adjoint orbit $\mathcal{O}_\lambda$, a complete flag manifold $\mathrm{Fl}(n)$ equipped with the monotone form $\omega_\lambda$. Fix the one parameter family of Lagrangian GC tori $L_m(t)$ over $I_m(t)$ for $0 \leq t < 1$ in $\mathcal{O}_\lambda$ as in~\eqref{IMT}.

We define the split leading term equation arising from the potential function of $L_m(t)$.
Let
$$
k := \left\lceil  n/2 \right\rceil,
$$
that is $n = 2k -1$ or $2k$.
We denote
\begin{equation}\label{eq:fbij}
f^\frak{b}_{(i, j)} ({\mathbf{y}}) := \partial_{(i,j)} W({\mathbf{y}}) \cdot T^{- \nu}
\quad \mbox{where $\displaystyle \nu := \frak{v} \left(\partial_{(i,j)} W({\mathbf{y}})\right)$}.
\end{equation}
We also denote
$f^{\vphantom{\frak{b}}}_{(i,j)} := f^\frak{b}_{(i,j)}$ when $\frak{b} = 0$.
For our purpose, it will suffice to take a bulk-deformation parameter of the type$\colon$
\begin{equation}\label{bulkparak2}
\frak{b} :=  \sum_{i \geq k} \frak{b}^\textup{hor}_{i, i+1} \cdot \scr{D}^\textup{hor}_{i, i+1}
 + \sum_{j \geq k} \frak{b}^\textup{ver}_{j+1, j} \cdot \scr{D}^\textup{ver}_{j+1, j}.
\end{equation}
Note that horizontal (resp. vertical) facets in~\eqref{bulkparak2} are supported outside of the maximal diagonal (of $(k \times k)$-size) embedded in $\Gamma(n)$.
Therefore we have
$$
\begin{cases}
c^\textup{hor}_{i, i+1} := \exp \left(\frak{b}^\text{hor}_{i,i+1} = 0 \right) = 1 \quad \mbox{for } i < k, \\
c^\textup{ver}_{j+1, j} := \exp \left(\frak{b}^\text{ver}_{j+1,j}= 0 \right) = 1  \quad \mbox{for } j < k.
\end{cases}
$$
\begin{definition}\label{splittingleadingtermequa}
Let $2 \leq m \leq k$ be given.
The \emph{split leading term equation} (abbreviated as \emph{SLT-equation}) \emph{of}
$\Gamma(n)$ \emph{associated with} $B(m)$ is the system
of the \emph{$\C$-valued} equations given by
\begin{equation}\label{splitleadingtermequ}
\begin{cases}
\ell^{\frak{b}, m}_{(i, j)} ({\mathbf{y}}) = 0 \quad &\mbox{for all $(i,j) \in \Gamma(n)$} \\
\ell^m_{(l)} ({\mathbf{y}}) = 0 \quad &\mbox{for all $l$ with $1 \leq l \leq m$}
\end{cases}
\end{equation}
where we define
\begin{eqnarray}
\ell^{\frak{b}, m}_{(i, j)}({\bf y})
&: = &- c^\textup{ver}_{j+1, j} \cdot \frac{y_{i, j+1}}{y_{i, j}}
- c^\textup{hor}_{i-1, i} \cdot \frac{y_{i-1, j}}{y_{i,j}} + c^\textup{hor}_{i, i+1} \cdot \frac{y_{i,j}}{y_{i+1, j}}
+ c^\textup{ver}_{j, j-1} \cdot \frac{y_{i, j}}{y_{i, j-1}},
\label{pabmijy}\\
\ell^m_{(l)} ({\mathbf{y}}) &:= &(-1)^{m+1-l} \cdot \frac{y_{l,m+1}}{y_{m,m}} + \frac{y_{m,m}}{y_{m+1,l}}.\label{pamly}
\end{eqnarray}
Here we split the system into two, one on $B(m)$ and the other outside $B(m)$, with the following
constraints:
\begin{enumerate}
\item For $\ell^{\frak{b}, m}_{(i, j)} ({\mathbf{y}}) = 0$ with $(i,j ) \in B(m)$, we set
\begin{equation}\label{settingupwithin}
\begin{cases}
y_{m,s} := 0 &\mbox{ for $s > m$} \\
y_{r,m} := \infty &\mbox{ for $r > m$}.
\end{cases}
\end{equation}
\item  For $\ell^{\frak{b}, m}_{(i, j)} ({\mathbf{y}}) = 0$ with $(i,j ) \in \Gamma(n) \backslash B(m)$, we set
\begin{equation}\label{settingup}
\begin{cases}
&y_{r,m} := \infty \quad \mbox{for } r < m, \quad y_{m, s} := 0 \quad \mbox{for } s < m\\
&y_{\bullet,0} := \infty, \,\, y_{0, \bullet} := 0, \,\, y_{s, n+ 1 - s} := 1 \quad \mbox{for } 1 \leq s \leq n \\
\end{cases}
\end{equation}
\end{enumerate}
\end{definition}

We explain the implication of the above constraints in writing the SLT-equation of $\Gamma(n)$ associated with $B(m)$.
Cutting the boxes $B(m) \backslash \{ (m,m) \}$ off from the diagram $\Gamma(n)$, we have decomposed
blocks as in Figure~\ref{Decomposition of the gradient of the potential function}.

For an index $(i, j) \in B(m)$, ignoring the terms containing the variables not in $B(m)$, $\ell^{\frak b, m}_{(i,j)} (\mathbf{y})$ consists of the initial terms of the logarithmic derivative of $W^\frak{b}$.
Also, for an index $(i, j) \in \Gamma(n) \backslash B(m)$, ignoring the terms containing the variables within $B(m) \backslash \{ (m,m) \}$, $\ell^{\frak b, m}_{(i,j)} (\mathbf{y})$ consists of the initial terms of the logarithmic derivative of $W^\frak{b}$ whose formula is in~\eqref{thegradientofbulkdeformedpotential}.
In addition to them, we solve the equation $\ell^m_{(l)} ({\mathbf{y}}) = 0$ for the diagonal variables
$y_{s,n+1-s}$ adjacent to $B(m)$. It will be explained in ~\eqref{yjm+1} why the latter equation occurs.

Now introduce the following linear orders on $\Gamma(n)$.
\begin{definition}
We say $(i,j) \prec_{\mathrm{hor}} (i',j')$ on $\Gamma(n)$ if one of the following alternatives holds
\begin{enumerate}
\item $i+j < i' + j'$,
\item $i+j = i' + j'$ and $i < i'$.
\end{enumerate}
We also similarly define $\prec_{\mathrm{ver}}$ by replacing $i <i'$ by $j < j'$ in (2) above.
\end{definition}

Existence of a solution for the SLT-equation~\eqref{splitleadingtermequ} will guarantee that
the assumption of the following lemma holds. We will repeatedly employ it in order to promote a
solution in $\C^*$ to that in $\Lambda_U$ for the critical point equation of the (bulk-deformed)
potential function.
\begin{lemma}\label{extensionlemma}
Suppose we are given $y_{i-1, j} \in \Lambda_U \cup \{ 0 \}$, $y_{i,j-1} \in \Lambda_U \cup \{ \infty\}$ and $y_{i, j}, y_{i,j+1}$. If there is a \emph{non-zero}
complex solution $y^\C_{i+1, j}$ of $\ell^{\frak{b},m}_{(i, j)} ({\mathbf{y}}) = 0 \mod T^{>0}$ for some $c^\textup{ver}_{j+1,j}, c^\textup{hor}_{i-1,i}, c^\textup{hor}_{i,i+1}$ and $c^\textup{ver}_{j,j-1}$ in $\Lambda_U$,
then $\ell^{\frak{b},m}_{(i, j)}({\mathbf{y}}) = 0$ has a unique solution
$y_{i+1,j} \in \Lambda_U$ such that $y_{i+1,j} \equiv y_{i+1,j}^\C \mod T^{>0}$.
The same holds by changing the role of $y_{i+1, j}$ and $y_{i,j+1}$.
\end{lemma}

\begin{proof}
We observe that $(i+1,j)$ is the maximal with respect to the order $\prec_{\mathrm{hor}}$
among the 4 points $(i-1,j), \, (i,j-1), \, (i, j+1), \, (i+1, j)$. Then
the proof immediately follows from the equation $\ell^{\frak{b},m}_{(i, j)} ({\mathbf{y}}) = 0$.
It shows that $y_{i+1,j}$ can be expressed as a rational function
in the other variables by the formula~\eqref{pabmijy}.
\end{proof}

\vspace{-0.1cm}
\begin{figure}[ht]
	\scalebox{0.65}{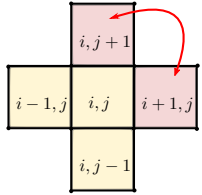}
	\caption{\label{graphicdeslemma} Graphical description of Lemma~\ref{extensionlemma}.}	
\end{figure}

The following lemma is also an immediate consequence of the valuation analysis of
the equation $\ell^{\frak{b},m}_{(i, j)} ({\mathbf{y}}) = 0$.
\begin{lemma}\label{extensionlemma2}
If in addition to the assumption of Lemma~\ref{extensionlemma} we assume 
\begin{itemize}
\item $c^\textup{ver}_{j+1,j}, c^\textup{hor}_{i-1,i}, c^\textup{hor}_{i,i+1}$ and $c^\textup{ver}_{j,j-1}$ are non-zero complex numbers
\item $
\frak{v}_T \left(y^{\phantom{\C}}_{i,j} - y^\C_{i,j} \right), \,
\frak{v}_T \left(y^{\phantom{\C}}_{i-1,j} - y^\C_{i-1,j} \right) \text{ and }
\frak{v}_T \left(y^{\phantom{\C}}_{i,j-1} - y^\C_{i,j-1} \right) > \lambda.
$
\end{itemize}
Then $\frak{v}_T (y_{i+1,j}) = \lambda$ if and only if $\frak{v}_T (y_{i,j+1}) = \lambda$.
\end{lemma}

Now we ready to state the main theorem of this section.
\begin{theorem}\label{splitleadingtermequationimpliesnonzero}
If there exist non-zero complex numbers $c^{\textup{hor}, \C}_{i, i+1}$ and $c^{\textup{ver}, \C}_{j+1, j}$ ($i, j \geq k$)
for which the SLT-equation~\eqref{splitleadingtermequ} admits a solution
$$
\mathbf{y}^\C := (y^\C_{i,j} \in \C^* \mid (i,j) \in \Gamma(n) \backslash B(m) \cup \{ (m,m)\}),
$$
then there exists a bulk-deformation $\frak{b}$
(depending on $m$ and $t$) of the form~\eqref{bulkparak2} such that
\begin{enumerate}
\item The bulk-deformed potential $W^\frak{b}({\mathbf{y}})$ has a critical point $(y_{i,j} \in \Lambda_U \mid (i,j) \in \Gamma(n))$ satisfying
$$
y^\C_{i,j} \equiv y^{\phantom{\C}}_{i,j} \mod T^{>0} \quad \text{for } (i,j) \in \Gamma(n) \backslash B(m) \cup \{ (m,m) \}.
$$
\item Also,
$\exp (\frak{b}^{\textup{hor} \phantom{, \C}}_{i,i+1})\equiv c^{\textup{hor}, \C}_{i,i+1},
\, \exp (\frak{b}^{\textup{ver} \phantom{, \C}}_{j+1,j})\equiv c^{\textup{ver}, \C}_{j+1, j}
 \mod T^{>0}$.
\end{enumerate}
\end{theorem}

\subsection{Solving SLT-equation with $\frak{b} = 0$ within $B(m)$}\label{symmetriccomplexsolinbm}

The goal of the section is to find a ``{symmetric}" complex solution of
\begin{equation}\label{paijmBm}
\ell^{0,m}_{(i, j)}(\mathbf{y})= - \frac{y_{i, j+1}}{y_{i, j}} - \frac{y_{i-1, j}}{y_{i,j}} + \frac{y_{i,j}}{y_{i+1, j}} + \frac{y_{i, j}}{y_{i, j-1}} = 0 \quad \text{ for all } (i, j) \in B(m).
\end{equation}
Notice that with the choice we made in~\eqref{bulkparak2}, $\frak{b}^\textup{hor}_{i,i+1} = 0$ (resp. $\frak{b}^\textup{ver}_{j+1,j} = 0$) for $i < k$ (resp. $j < k$).
This in turn implies
$\ell^{\frak{b},m}_{(i, j)}(\mathbf{y})$ given in~\eqref{pabmijy} actually coincides with
$\ell^{0,m}_{(i, j)}(\mathbf{y})$ for $(i,j) \in B(m)$.

\begin{lemma}\label{minussol}
Let $\mathbf{y}^\C := ({y}^\C_{i,j} \in \C^* \mid (i,j) \in B(m) )$ be a
solution  of~\eqref{paijmBm}. Then
$$
\widetilde{y^\C_{i,j}} := c \cdot y^\C_{i,j} \quad \mbox{for }\, (i,j) \in B(m)
$$
is also a solution of~\eqref{paijmBm} for any non-zero complex number $c$.
\end{lemma}

\begin{proof}
The lemma immediately follows from $\ell^{0,m}_{(i,j)} (\mathbf{y}) = \ell^{0,m}_{(i,j)} (c \cdot \mathbf{y})$.
\end{proof}

\begin{lemma}\label{symmetricsol}
There exists a solution $\mathbf{y}^\C := \left(y^\C_{i,j} \in \C^* \mid i+j \leq m+1\right)$ of the system of equations
\begin{equation}\label{paijmbfyijm}
\ell^{0,m}_{(i,j)} (\mathbf{y}) = 0 \quad \mbox{for }\, i + j \leq m
\end{equation}
such that it is symmetric in that
\begin{equation}\label{symmetricproperty}
y^\C_{i,j} = \left(y^\C_{j,i}\right)^{-1}.
\end{equation}
\end{lemma}

\begin{proof}
It is straightforward to check
\begin{equation}\label{solutionofpamij}
\displaystyle y^\C_{i,j} :=
\begin{cases}
1 \quad &\mbox{for } i = j \\
\displaystyle \prod_{r=0}^{j - i - 1} (2i + 2r) \quad &\mbox{for } i < j \\
\displaystyle \prod_{r=0}^{i - j - 1} (2j + 2r)^{-1} \quad &\mbox{for } i > j
\end{cases}
\end{equation}
is a symmetric solution for~\eqref{paijmbfyijm}.
\end{proof}

We are ready to prove the existence of a symmetric solution in the sense of~\eqref{symmetricproperty}.
\begin{proposition}\label{propsymmetricsol}
There exists a symmetric solution $\mathbf{y}^\C := \left({y}^\C_{i,j} \in \C^* \mid (i,j) \in B(m)\right)$
of~\eqref{paijmBm} satisfying $y^\C_{i,i} = \pm 1$ for all $1 \leq i \leq m$.
\end{proposition}

\begin{proof}
We start with a solution $y^\C_{i,j} \in \C^*$ for $i+j \leq m+1$ given in
Lemma~\ref{symmetricsol}. By the choice, it satisfies $y^\C_{i,i} = 1$ for all $i$ with $2i \leq m$.
For the remaining indices $(i,j)$ with $i + j > m + 1$, we put
\begin{equation}\label{fijfullm}
y_{i,j}^\C := (-1)^{i+j-m-1} \, y^\C_{m+1-j, m+1-i}.
\end{equation}
We now show that this choice gives rise to a solution for~\eqref{paijmBm}.

For $(i,j) \in B(m)$ with $i + j \geq m + 2$, it is straightforward to check
\begin{align*}
\ell^{0,m}_{(i,j)} (\mathbf{y}^\C)  = - \ell^{0,m}_{(m+1-j,m+1-i)} (\mathbf{y}^\C) = 0
\end{align*}
from the choice made in Lemma~\ref{symmetricsol} for $i+j\leq m+1$. For the indices
$(i,j)$ with $i + j = m + 1$, the definition \eqref{fijfullm} implies
$y^\C_{i, j+1} = -y^\C_{i-1, j}$ and hence
\begin{align*}
\ell^{0,m}_{(i,j)} (\mathbf{y}^\C)  &= - \frac{y^\C_{i, j+1}}{y^\C_{i,j}} - \frac{y^\C_{i-1,j}}{y^\C_{i,j}} + \frac{y^\C_{i,j}}{y^\C_{i+1,j}} +\frac{y^\C_{i,j}}{y^\C_{i,j-1}} = \frac{y^\C_{i-1, j}}{y^\C_{i,j}} - \frac{y^\C_{i-1,j}}{y^\C_{i,j}} + \frac{y^\C_{i,j}}{y^\C_{i+1,j}} -\frac{y^\C_{i,j}}{y^\C_{i+1,j}} = 0.
\end{align*}
The symmetry ~\eqref{symmetricproperty} for $(i,j)$ with $i + j \leq m + 1$ and the definition \eqref{fijfullm} also immediately imply the symmetry of the solution and $y_{i,i} = \pm 1$
for all $1 \leq i \leq m$.
\end{proof}

\begin{corollary}\label{corsymmetricsol}
For any $c \in \C^*$, there is a solution $\mathbf{y}^\C := \left({y}^\C_{i,j} \in \C^* \mid (i,j)
\in B(m)\right)$ of~\eqref{paijmBm} such that
\begin{enumerate}
\item $y^\C_{i,j} \cdot y^\C_{j,i} = c^2$ .
\item $y^\C_{m,m} = c^{\phantom{\C}}$
\item $y^\C_{i,i} = \pm \, c^{\phantom{\C}}$ for any $1 \leq i < m$.
\end{enumerate}
\end{corollary}

\begin{proof}
The component $y^\C_{m,m}$ of a solution from Proposition~\ref{propsymmetricsol} is either $1$ or $-1$. By multiplying by $\pm c$ to the solution found therein and using Lemma \ref{minussol},
we produce another solution that satisfy $(1), (2)$ and $(3)$.
\end{proof}

\subsection{Determination of $\mathbf{y}$ inside $B(m)$}\label{insidebm}

\begin{hypothesis}
We assume that the SLT-equation of $\Gamma(n)$ associated with $B(m)$ has a solution for some non-zero complex numbers $c^{\textup{hor},\C}_{i, i+1}$'s and $c^{\textup{ver},\C}_{j+1, j}$'s
for $i, j \geq k= \left\lceil  n/2 \right\rceil$.
\end{hypothesis}

Consider the variables
\begin{equation}\label{seedabovebox}
y^\C_{m,m} \textup{ and } y^\C_{i, m+1} \, \text{with } 1 \leq i \leq m
\end{equation}
where $y^\C_{i,j} \in \C^*$ is the $(i,j)$-th component of a solution of the SLT-equation.
In order to emphasize that~\eqref{seedabovebox} has been pre-determined, we denote the values by
\begin{equation}\label{eq:dij}
d^{\phantom{\C}}_{i,j} := y_{i,j}^\C, \quad (i,j) = (m,m) \text{ or } (i,m+1)
\text{ with } 1 \leq i\leq m.
\end{equation}
Setting it as the initial part for a solution and using Lemma~\ref{extensionlemma}, we promote it to a solution of~\eqref{pabmijy} over $\Lambda_U$ \emph{in the increasing order of $(i,j) \in B(m)$
with respect to the order $\prec_{\mathrm{hor}}$}. For a pictorial outline of this section, see Figure~\ref{Stepsinside}.

\begin{figure}[h]
	\scalebox{0.9}{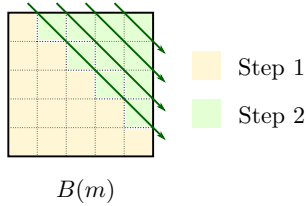}
	\caption{\label{Stepsinside} Pictorial outline of Section~\ref{insidebm}.}	
\end{figure}

\smallskip

\noindent{\bf Step 1. $(i, j) \in B(m)$ with $i + j \leq m + 1$:} We begin by taking
$y^{\phantom{\C}}_{i,j} := {y^\C_{i,j}} \in \C^* \subset \Lambda_U $ where $(y^\C_{i,j} \mid (i,j) \in B(m))$
is a solution satisfying $y^\C_{m,m} = d^{\phantom{\C}}_{m,m}$ from Corollary~\ref{corsymmetricsol}
for all $(i,j)$ with $i + j \leq m+1$.

\smallskip

\smallskip

\noindent{\bf Step 2. $(i, j) \in B(m)$ with $m + 2 \leq i + j \leq 2m$:}
We first compute
$$
\partial_{(i,j)} W({\mathbf{y}})
= \left(- \frac{y_{j-1,m}}{y_{j,m}} + \frac{y_{j,m}}{y_{j+1,m}} + \frac{y_{j,m}}{y_{j,m-1}} \right) T^{1-t}
+  \left( - \frac{y_{j,m+1}}{y_{j,m}} \right) T^{1+(m-j)t}
$$
from which we have
$$
\displaystyle \nu = \frak{v} \left(\partial_{(i,j)} W({\mathbf{y}})\right)
=1-t.
$$
Then by definition of $f_{(j,m)}$ in \eqref{eq:fbij} with ${\frak b} = 0$, we have
$$
f_{(j,m)} (\mathbf{y})  = \left(- \frac{y_{j-1,m}}{y_{j,m}} + \frac{y_{j,m}}{y_{j+1,m}} + \frac{y_{j,m}}{y_{j,m-1}} \right)
 +  \left( - \frac{y_{j,m+1}}{y_{j,m}} \right) T^{(m-j+1)t}.
$$

For an index $j$ with $1 \leq j < m$, we decompose $f_{(j,m)} (\mathbf{y})$:
$$
f_{(j,m)} (\mathbf{y}) = h_{(j,m)}^{(1)} (\mathbf{y})  + h_{(j,m)}^{(2)} (\mathbf{y}) \,\, T^{(m-j+1)t}
$$
where
$$
\begin{cases}
\displaystyle h_{(j,m)}^{(1)} (\mathbf{y}) := \left(- \frac{y_{j-1,m}}{y_{j,m}} + \frac{y_{j,m}}{y_{j+1,m}}
+ \frac{y_{j,m}}{y_{j,m-1}}\right) + \left( {\color{blue} - a_j \frac{y_{j,m}}{y_{j+1,m}} - a_j \frac{y_{j,m}}{y_{j,m-1}}} \right) T^{(m-j+1)t} \\
\displaystyle h_{(j,m)}^{(2)} (\mathbf{y}) := - \frac{y_{j,m+1}}{y_{j,m}} {\color{blue} + a_j \frac{y_{j,m}}{y_{j+1,m}}
+ a_j \frac{y_{j,m}}{y_{j,m-1}}}.
\end{cases}
$$
Here we recall $y_{0,m} \equiv 0$ from \eqref{settingup}. This term
appears only for the case $j=1$.

We will solve $f_{(j,i)}(\mathbf{y}) = 0$ on $B(m)$ inductively
in the increasing order of $\prec_{\mathrm{hor}}$.
For this purpose, we utilize the following sufficient conditional equation.

\begin{lemma} A solution of the system
\begin{align}\label{solutdecompos}
\begin{cases}
\displaystyle h_{(j,m)}^{(1)} ({\mathbf{y}}) = 0 \\
\displaystyle h_{(j,m)}^{(2)} ({\mathbf{y}}) - a_j \cdot f_{(j,m)}({\mathbf{y}}) = - \frac{y_{j,m+1}}{y_{j,m}}  + a_j \left( \frac{y_{j-1,m}}{y_{j,m}}  + \frac{y_{j,m+1}}{y_{j,m}} \, T^{(m-j+1)t}\right) = 0
\end{cases}
\end{align}
is also a solution of $f_{(j,m)}(\mathbf{y}) = 0$.
\end{lemma}

\begin{proof} The proof is omitted.
\end{proof}

As the first step of the induction, suppose that we are given a solution
$(y_{r,s} \in \Lambda_U \mid (r,s) \in B(m) , \, r+s \leq m+j)$
of $f_{(r,s)}(\mathbf{y}) = 0$
for all $(r,s) \in B(m)$ and $r+s \leq m +j$ such that
$$
\frak{v}_T (y^{\phantom{\C}}_{r,s} - y_{r,s}^\C) \geq (m - j + 2)t
$$
as the induction hypothesis.
Since each $y_{r,s}^\C$ is non-zero by our choice in \eqref{solutionofpamij}, we derive
$$
y^\C_{i,j} \neq 0\, \text{ for all } \, (i,j) \in B(m)
$$
from Corollary~\ref{corsymmetricsol}.
Therefore we have ${y^\C_{j-1,m}}/{y^\C_{j,m}} \neq 0$ for $j \geq 2$ in particular. Furthermore
thanks to the second equation of \eqref{solutdecompos}, the equation
$$
\begin{cases}
h_{(j,m)}^{(2)} ({\mathbf{y}}) - a_j \cdot f_{(j,m)}({\mathbf{y}}) = 0\\
y_{j,m+1} = d_{j,m+1}
\end{cases}
$$
uniquely determines an element $a_j \in \Lambda_U$. Then the equation
$h_{(j,m)}^{(1)} (\mathbf{y}) = 0$ gives rise to
$$
\frac{y^\C_{j,m}}{y_{j+1,m}} \left(1 - a_j \, T^{(m-j+1)t} \right) \equiv \left( \frac{y^\C_{j-1,m}}{y^\C_{j,m}} - \frac{y^\C_{j,m}}{y^\C_{j,m-1}} \right) + a_j \frac{y^\C_{j-1,m}}{y^\C_{j,m}} \, T^{(m-j+1)t}  \mod T^{>(m-j+1)t}.
$$
Recall that $y_{r,s}^\C$'s arising from Corollary~\ref{corsymmetricsol} are non-zero and satisfy
\begin{equation}\label{eq:ell0mjm}
\ell^{0,m}_{(j,m)}(\mathbf{y}^\C) = - \frac{y^\C_{j-1,m}}{y^\C_{j,m}}  + { \frac{y^\C_{j,m}}{y^\C_{j+1,m}} + \frac{y^\C_{j,m}}{y^\C_{j,m-1}}} = 0.
\end{equation}
Therefore
$y_{j+1,m} \in \Lambda_U$ is uniquely determined by induction with respect to $y^\C_{i,\ell}$'s
with $i+\ell \leq j+m$ and satisfies
$$
\frak{v}_T (y^{\phantom{\C}}_{j+1,m} - y_{j+1,m}^\C) = (m-j+1)t.
$$

Next as the second step of the induction along each diagonal, suppose that
$$
\{(y_{r,s} \in \Lambda_U) \mid (r,s) \in B(m),\, (r,s) \prec_{\mathrm{hor}} (i+j, m-i+1)\}
$$
are given, and satisfies
$$
\frak{v}_T (y^{\phantom{\C}}_{r,s} - y^\C_{r,s})  = (m - j + 1) t
$$
for $(r,s)$ with $r+s = m+j+1$. We will then determine for $(r,s) = (i+j+1,m-i)$ therefrom.
By Lemma~\ref{extensionlemma} and~\ref{extensionlemma2}, the equation $f_{(r,s-1)}(\mathbf{y}) = 0$ uniquely determines $y_{r+1,s-1} \in \Lambda_U$
so that
\begin{equation}\label{valofy2}
\frak{v}_T (y^{\phantom{\C}}_{r+1,s-1} - y^\C_{r+1,s-1}) = (m - j + 1)t.
\end{equation}

In order to find $y_{m+1,j}$, we convert $f_{(m,j)} (\mathbf{y}) = 0$ into
$h_{(m,j)}^{(2)} (\mathbf{y}) = 0$ by inserting the previously determined $y_{i,j}$'s.
For $0 \leq i+j\leq m-1$, thanks to~\eqref{valofy2}, we may set
\begin{equation}\label{eq:yCr+1s-1}
y^{\phantom{\C}}_{r+1,s-1} \equiv {y}^\C_{r+1,s-1} + A_{r-j} \cdot T^{(m - j + 1) t} \mod T^{> (m - j + 1) t}
\end{equation}
where $A_{r-j} \in \C^*$ for all $r \leq i+j$.

\begin{lemma}\label{aibi2}
A recurrence relation for $A_i$'s to satisfy is given by
$$
\begin{cases}
\displaystyle A_0 &= - a_j \cdot \frac{(y^\C_{j+1,m})^2}{(y^\C_{j,m})^2} \cdot y^\C_{j-1,m} \\
\displaystyle A_{i} &= - \frac{\,\, \left({y}^\C_{i+j+1,m-i} \right)^2 }{\,\, \left({y}^\C_{i+j,m-i} \right)^2} \, A_{i-1}.
\end{cases}
$$
\end{lemma}
\begin{proof} This follows from a straightforward valuation analysis on $\partial_{(i+j,m-i)}W(\mathbf{y}) = 0$
with the insertion of \eqref{eq:yCr+1s-1} into $\mathbf{y}$ for $(r,s) = (i+j+1,m-i)$ .
\end{proof}

Using Lemma~\ref{aibi2}, \eqref{fijfullm}, \eqref{eq:ell0mjm} with $(r+1,s-1) = (m+1,j)$ and \eqref{eq:yCr+1s-1}, we derive
\begin{align*}
f_{(m,j)} (\mathbf{y}) &=  \left(- \frac{y_{m,j+1}}{y_{m,j}} - \frac{y_{m-1,j}}{y_{m,j}} + \frac{y_{m,j}}{y_{m,j-1}} \right) +
+ \frac{y_{m,j}}{y_{m+1,j}} T^{(m-j+1)t} \\
& \equiv \left(- \frac{y^\C_{m,j+1}}{y^\C_{m,j}} - \frac{y^\C_{m-1,j}}{y^\C_{m,j}} + \frac{y^\C_{m,j}}{y^\C_{m,j-1}} \right) +
\left(- \frac{A_{m-j}}{y^\C_{m,j-1}} + \frac{y^\C_{m,j}}{y^\C_{m+1,j}}\right) T^{(m-j+1)t}
\\
& \equiv \left(- \frac{A_{m-j}}{y^\C_{m,j-1}} + \frac{y^\C_{m,j}}{y^\C_{m+1,j}}\right) T^{(m-j+1)t}
\end{align*}
modulo ${T^{>(m-j+1)t}}$.
By applying Lemma~\ref{aibi2} iteratively following the increasing order of $\prec_{\mathrm{hor}}$,
we further see the last formula becomes
$$
- \frac{A_{m-j}}{y^\C_{m,j-1}} + \frac{y^\C_{m,j}}{y^\C_{m+1,j}}=  (-1)^{m-j+1} a_j \, \frac{y^\C_{m,j}}{y^\C_{m,j-1}}
+ \frac{y^\C_{m,j}}{y^{\phantom{\C}}_{m+1,j}}.
$$
Combining the above discussion, we have derived
\begin{equation}\label{pajm2}
h_{(m,j)}^{(2)}(\mathbf{y}) = y^\C_{m,j}  \left( (-1)^{m-j+1} a_j \, \frac{1}{y^\C_{m,j-1}} + \frac{1}{y^{\phantom{\C}}_{m+1,j}} \right) + k_{(m,j)}^{(2)} (a_1, \cdots, a_j) = 0.
\end{equation}
for some $k_{(m,j)}^{(2)} (a_1, \cdots, a_j) \in \Lambda_+$.

We then derive from this and~\eqref{pamly} Corollary~\ref{corsymmetricsol} (1)
\begin{equation}\label{yjm+1}
y_{m+1,j} \equiv (-1)^{m-j} \, \frac{y_{m,j-1}}{a_j} \equiv (-1)^{m-j} \, \frac{(d_{m,m})^2}{a_j \cdot y_{j-1,m}} \equiv (-1)^{m-j} \frac{(d_{m,m})^2}{y_{j,m+1}} \mod{ T^{>0}},
\end{equation}
which explains why the equation $\ell^m_{(j)} (\mathbf{y}) = 0$ in the system \eqref{splitleadingtermequ} appears. In other words,~\eqref{pamly} provides a sufficient condition to solve $y_{m+1,j}$ over $\Lambda_U$ in~\eqref{pajm2}.

Finally, we convert $f_{(m,m)}(\mathbf{y}) = 0$ into $h_{(m,m)}^{(2)}(\mathbf{y}) = 0$ as follows. Inserting $y_{m-1,m}, y_{m,m}, y_{m,m-1}$ and $y_{m,m+1} =  d_{m,m+1}$ into $f_{(m,m)}(\mathbf{y}) = 0$, we derive
\begin{equation}\label{pamm2}
h_{(m,m)}^{(2)}(\mathbf{y}) = \left( - \frac{y^{\phantom{\C}}_{m,m+1}}{y^\C_{m,m}} + \frac{y^\C_{m,m}}{y^{\phantom{\C}}_{m+1, m}} \right) + k_{(m,m)}^{(2)}(\textbf{a}) = 0
\end{equation}
for some $k_{(m,m)}^{(2)}(\textbf{a}) \in \Lambda_+$. We obtain
$$
y_{m+1, m} \equiv \frac{(y_{m,m})^2}{y_{m,m+1}} \mod T^{>0}
$$
and determine $y_{m+1,m}$ in $\Lambda_U$.

We summarize the above discussion into the following.
\begin{proposition}
For any tuple $(d_{m,m}, d_{1, m+1}, \cdots, d_{m,m+1})$ of non-zero complex numbers, we can find
$y_{i,j} \in \Lambda_U$ for all $(i,j) \in B(m)$ and $(i,m+1), \, (m+1,i)$ with $1 \leq i \leq m$
so that they satisfy
\begin{enumerate}
\item $y_{m,m} \equiv d_{m,m} \mod T^{>0}$,
\item $y_{i, m+1} = d_{i, m+1}$ for each $i = 1, \cdots, m$,
\item $f_{(i,j)} (\mathbf{{y}}) = 0$ \, for $(i,j) \in B(m)$,
\item $(-1)^{m+1-l} \, \frac{y_{l,m+1}}{y_{m,m}} +  \frac{y_{m,m}}{y_{m+1,l}} \equiv 0 \mod T^{>0}$.
\end{enumerate}
\end{proposition}

\subsection{Solving for $(\frak b;\mathbf{y})$ outside $B(m)$} \label{outsideofbm}

In this section, we discuss how to determine a bulk-deformation parameter $\frak{b}$ in~\eqref{bulkparak2} from $c^{\textup{ver}, \C}_{i+1,i}$'s and $c^{\textup{hor}, \C}_{j,j+1}$'s and extend it to a solution over $\Lambda_U$ from $y^\C_{i,j}$'s under the following hypothesis.
\begin{hypothesis}
Suppose the hypothesis of Theorem~\ref{splitleadingtermequationimpliesnonzero}, i.e.,
that we are given a complex solution
$
\left(y^\C_{i,j} \in \C^* \mid  (i,j) \in \Gamma(n) \backslash B(m) \cup \{(m,m)\}\right)
$
for~\eqref{splitleadingtermequ} together with non-zero complex numbers $c^{\textup{ver}, \C}_{i+1,i}$'s and $c^{\textup{hor}, \C}_{j,j+1}$'s.
\end{hypothesis}

In the setup of~\eqref{settingup}, to apply Lemma~\ref{extensionlemma}, we rewrite the equation $\ell^{\frak{b},m}_{(i,j)}  ({\mathbf{y}}) =0$ into
\begin{equation}\label{equ_sioslatelemma}
\begin{cases}
c^\textup{hor}_{i, i+1} \cdot \frac{1}{y_{i+1, j}} = h^{\frak{b},m}_{(i,j)}  ({\mathbf{y}})
\quad &\mbox{if \,} i \geq j \\
c^\textup{ver}_{j+1, j} \cdot {y_{i, j+1}} = h^{\frak{b},m}_{(i,j)}  ({\mathbf{y}})
\quad &\mbox{if \,} i < j
\end{cases}
\end{equation}
where
\begin{equation}\label{tildepabijy}
h^{\frak{b},m}_{(i,j)}  ({\mathbf{y}}) :=
\begin{cases}
\displaystyle - c^\textup{ver}_{j, j-1} \cdot \frac{1}{y_{i, j-1}} + \frac{1}{(y_{i, j})^2} \left( c^\textup{ver}_{j+1, j} \cdot {y_{i, j+1}}+ c^\textup{hor}_{i-1, i} \cdot {y_{i-1, j}} \right) \quad &\mbox{if \,} i \geq j \\
\displaystyle - c^\textup{hor}_{i-1, i} \cdot {y_{i-1, j}} + (y_{i,j})^2 \left( c^\textup{hor}_{i, i+1} \cdot \frac{1}{y_{i+1, j}} + c^\textup{ver}_{j, j-1} \cdot \frac{1}{y_{i, j-1}} \right) \quad &\mbox{if \,} i < j.
\end{cases}
\end{equation}
As one can see, both sides of~\eqref{equ_sioslatelemma} depend on bulk-parameters.

Now, we introduce a new coordinate system $( z_{i,j} \mid (i,j) \in \Gamma(n) \backslash B(m) \cup \{(m,m)\} )$ with respect to which the system of equations
$
h^{\frak{b},m}_{(i,j)}  ({\mathbf{y}})  = 0,
$
does \emph{not} depend on the choice of a bulk-deformation parameter $\frak{b}$ for
those $(i,j)$ satisfying $ i+j < n$. We set
\begin{equation}\label{coordinatechange}
\begin{cases}
\displaystyle z_{i+1, \bullet} := \left( \prod^{i}_{r=k} {c^\textup{hor}_{r, r+1} } \right)^{-1} y_{i+1,\bullet} \quad &\mbox{if \,} i \geq k \\
\displaystyle z_{\bullet, j+1} := \left( \prod^{j}_{r=k} c^\textup{ver}_{r+1, r} \right) y_{\bullet, j+1} \quad &\mbox{if \,} j \geq k \\
\displaystyle z_{i,j} := y_{i,j} \quad &\mbox{otherwise}
\end{cases}
\end{equation}
where $k = \left\lceil n/2 \right\rceil$.
Setting that the product over the empty set to be $1$ such as in
$
\prod_{r = k}^{k-1} c^\textup{ver}_{r+1, r} = 1,
$
the same inductive isolation procedure applied to $h^{\frak{b},m}_{(i,j)}  ({\mathbf{y}})  = 0$
under this coordinate system leads us to
the following Laurent polynomials:
\begin{equation}\label{tildepabijz}
k^{\frak{b},m}_{(i,j)}  ({\mathbf{z}})  :=
\begin{cases}
\displaystyle - \frac{1}{z_{i, j-1}} + \frac{1}{(z_{i, j})^2} \left( {z_{i, j+1}} + {z_{i-1, j}} \right)
\hskip0.7in \left( = \frac{1}{z_{i+1,j}}\right)  &\mbox{if \,} i \geq j, \, i + j < n \\
\displaystyle - {z_{i-1, j}} + ({z_{i,j}})^2 \left( \frac{1}{z_{i+1, j}} + \frac{1}{z_{i, j-1}} \right)
\hskip0.7in \left(=  z_{i, j+1}\right)  &\mbox{if \,} i < j, \, i + j < n \\
\displaystyle - \frac{1}{z_{i, j-1}} + \frac{1}{(z_{i, j})^2} \left( \left( \prod^{i-1}_{r=k} {c^\textup{hor}_{r, r+1} } \right)^{-1} + {z_{i-1, j}} \right) \left( = \prod^{i}_{r=k} {c^\textup{hor}_{r, r+1} } \right)
 &\mbox{if \,} i \geq j, \, i + j = n \\
\displaystyle - {z_{i-1, j}} + ({z_{i,j}})^2 \left( \left( \prod^{j-1}_{r=k} c^\textup{ver}_{r+1, r} \right)^{-1}
+ \frac{1}{z_{i, j-1}} \right)  \left(= \prod^{j}_{r=k} c^\textup{ver}_{r+1, r} \right)
 &\mbox{if \,} i < j, \,  i + j = n.
\end{cases}
\end{equation}
Here the variables inside the parentheses are the isolated variable
to be determined by $k^{\frak{b},m}_{(i,j)}  ({\mathbf{z}})$, respectively.
Observe the following$\colon$
\begin{itemize}
\item For any $(i,j)$ with $i + j < n$,~\eqref{tildepabijz} does \emph{not} involve any bulk-parameters.
\item A solution $(y_{i,j} \in \Lambda_U)$ of~\eqref{equ_sioslatelemma} exists if and only if a solution $(z_{i,j} \in \Lambda_U)$ exists.
\end{itemize}

Assume $m < \left\lceil n/2 \right\rceil =: k$. We will separately deal with the case where $m = \left\lceil n/2 \right\rceil$ later on.
Let
\begin{equation}
\begin{split}
\label{IndexsetforSeed} &\mcal{I}_{\textup{seed}} := \left\{ (m,m) \right\} \cup \left\{(s, m+1) \in \Gamma(n) \mid 1 \leq s \leq m \right\} \\
&\cup \left\{ ({m+1,m+1}), ({m+1,m+2}), \dots, (m+r, m+r), (m+r, m+r+1), \dots,  (k,k) \right\}
\end{split}
\end{equation}
and
\begin{align}
\label{IndexsetforSeedbackslashmm} &\mcal{I}_{\textup{initial}} :=  \mcal{I}_{\textup{seed}} \backslash \{(m,m)\}.
\end{align}
We start by taking $z^{\phantom{\C}}_{i,j} = y^{\phantom{\C}}_{i,j} := y^\C_{i,j}$ for $(i,j) \in \mcal{I}_\textup{initial}$ by~\eqref{coordinatechange} and fixing a complex solution from Corollary~\ref{corsymmetricsol} such that $c = y_{m,m}^\C$ within $B(m)$.
When extending a complex solution to that a solution over $\Lambda_U$, we shall have $z^{\phantom{\C}}_{i,j} = y^{\phantom{\C}}_{i,j} \in \C^*$ for any $(i,j) \in \mcal{I}_\textup{initial}$, while $y_{m,m} = c + \dots \in \Lambda_U$. 

\begin{remark}
We will define a \emph{seed} in Definition~\ref{definitionofseeds} to generate  a solution candidate for the SLT-equation. \eqref{IndexsetforSeed} is the collection of indices where the corresponding variables will be chosen as the initial step.
\end{remark}

By the argument in Section~\ref{insidebm}, the chosen element $z_{1,m+1} \in \C^*$ determines $z_{i+1,m-i+1} = y_{i+1,m-i+1} \in \Lambda_U$ for $1 \leq i \leq m$. In particular, we have $z_{m+1,1} = y_{m+1,1} \in \Lambda_U$.
We then proceed to the next diagonal given by $i+j = m+3$.
Moreover, we find $z_{i,j} = y_{i,j}$ in $\Lambda_U$ for $(i,j)$ with $i \leq m+1$ and $j \leq m+1$ in the diagonal satisfying~\eqref{yjm+1}.
Applying Lemma~\ref{extensionlemma} to~\eqref{tildepabijz}, we determine the other variables $z_{i,j} \in \Lambda_U$ in the diagonal $i+j = m+3$.
Namely, starting from $z_{2,m+1} \in \C^*$ and $z_{m+1,2} \in \Lambda_U$ determined in Section~\ref{insidebm}, we solve $z_{1,m+2} (\prec_{\mathrm{hor}} z_{2,m+1})$ and $z_{m+2, 1} (\prec_{\mathrm{ver}} z_{m+1, 2})$.
Notice that Lemma~\ref{extensionlemma} is applicable because of the existence of a complex solution of the SLT-equation.
Proceeding inductively, we obtain a solution $(z_{i,j} \in \Lambda_U \mid (i,j) \in \Gamma(n))$ for the system~\eqref{tildepabijz} of equations over $(i, j) \in \Gamma(n-1)$. 
Finally, we solve $c^\textup{hor}_{r,r+1}$ and $c^\textup{ver}_{r+1,r}$ to make the equations~\eqref{tildepabijz} over $(i,j)$ with $i + j = n$ hold. Thus, Theorem~\ref{splitleadingtermequationimpliesnonzero} is now verified when $m < \left\lceil n/2 \right\rceil$.

In this case when $n = 2k$ and $m = k$, we shall take $c = 1$ for $y^\C_{m,m}$. Corollary~\ref{corsymmetricsol} will give us the initial parts $y^\C_{i,j}$'s of $y^{\phantom{\C}}_{i,j}$'s for $(i,j) \in B(m)$. We then follow Section~\ref{insidebm} to extend to $z_{i,j} = y_{i,j}$'s over $\Lambda_U$ in $B(m)$. If one uses both $\frak{b}^\textup{ver}_{m+1,m}$ and $\frak{b}^\textup{hor}_{m,m+1} $ to deform $f_{(m,m)} = 0$, then we have two extra variables $c^\textup{ver}_{m+1,m}$ and $c^\textup{hor}_{m,m+1}$ in $f^\frak{b}_{(m,m)} = 0$. For our convenience, we choose $\frak{b}^\textup{hor}_{m,m+1} = 0$. Now, we need to take $\frak{b}^\textup{ver}_{m+1,m}$ so that $f^\frak{b}_{(m,m)} = 0$.
It yields that $1 = c^{\textup{ver},\C}_{m+1,m} = \exp (\frak{b}^{\textup{ver}\phantom{,\C}}_{m+1,m}) \mod T^{>0}$. After fixing $\frak{b}^{\textup{ver}\phantom{,\C}}_{m+1,m} \in \Lambda_+$, we obtain $z_{\bullet, m+1} = y_{\bullet, m+1}$ by solving $h^{\frak{b}, (2)}_{(\bullet, m)} = 0$ where
\begin{align*}
h^{\frak{b},(2)}_{(1,m)} (\mathbf{y}) &:= - \exp (\frak{b}^{\textup{ver}\phantom{,\C}}_{m+1,m}) \, \frac{y_{1,m+1}}{y_{1,m}} { + a_1 \frac{y_{1,m}}{y_{1,m-1}}}  = 0 \\
h^{\frak{b},(2)}_{(j,m)} (\mathbf{y}) &:= - \exp (\frak{b}^{\textup{ver}\phantom{,\C}}_{m+1,m}) \, \frac{y_{j,m+1}}{y_{j,m}} + a_j \left( { \frac{y_{j,m}}{y_{j+1,m}} + \frac{y_{j,m}}{y_{j,m-1}}} \right)= 0 \quad \mbox{for } j > 2.
\end{align*}
The remaining steps are similar.

\section{Solvability of split leading term equation}\label{secSolvabilityOfSplitLeadingTermEquation}\label{solvabilityofthesplitleadingtermequ}

This section aims to verify the assumption for Theorem~\ref{splitleadingtermequationimpliesnonzero} when the SLT-equation~\eqref{splitleadingtermequ} comes from the line segment $I_m \subset \Delta_\lambda$ in~\eqref{IMT}. To find its solution, we introduce the notion of \emph{seed} generating a candidate for  the solution and prove that there exists a ``good" choice of seeds such that the candidate is indeed a solution.

\subsection{Seeds}
\label{subsec:seeds}

Recall the definitions $\Gamma(n)$ and $B(m)$ from~\eqref{gammanbm}. Consider a nested sequence of fields of rational functions$\colon$
\begin{equation}\label{equ_nestseq}
\C \subset \C ({\mathbf{y}_{(1)}}) \subset \C ({\mathbf{y}_{(2)}}) \subset \dots \subset \C ({\mathbf{y}_{(n)}})
\end{equation}
where $\C({\mathbf{y}_{(n)}})$ is filtered by the variables in the diagonals of $\Gamma(n)$.
Set
\begin{eqnarray*}
\C ({\mathbf{y}_{(r)}}) & := & \C \left(\{(i,j) \in \Gamma(n) \backslash (B(m) \backslash \{ y_{m,m} \}) \mid (i,j) \prec_{\mathrm{hor}} (2r, 1)\}\right)\\
& = & \C\left(\bigcup_{s=2}^r \{y_{1,s-1}, \dots, y_{s-1,1} \} \backslash (B(m) \backslash \{ y_{m,m} \})  \right).
\end{eqnarray*}
Such a filtration will be called the \emph{diagonal filtration associated to} $\Gamma(n)$ and $B(m)$.
Notice that the Laurent polynomials appearing in (LHS) of~\eqref{splitleadingtermequ} are actually of the following types.
$$
\ell^{\frak{b}, m}_{(i,j)} (\mathbf{y}) \in \C({\mathbf{y}_{(i+j+1)}}), \quad \ell^m_{(l)} (\mathbf{y}) \in \C({\mathbf{y}_{(l+m+1)}} \cup \{y_{m,m}\}).
$$

We begin by the definition of a seed.

\begin{definition}\label{definitionofseeds}
A \emph{seed} of $\Gamma(n)$ \emph{associated with} $B(m)$ consists of the two data $(\mcal{I}, {\mathbf{d}}_{\mcal{I}})$.
\begin{itemize}
\item An $(n- m)$-tuple $\mcal{I}$ of double indices
$$
\mcal{I} = \{(m,m), (i_1,j_1), \cdots, (i_{n-m-1},j_{n-m-1})\} \subset
\{(m,m)\}\cup (\Gamma(n) \backslash B(m))
$$
such that one index $(i_k,j_k)$ is selected from each diagonal not intersecting $B(m)$.
\item An $(n - m)$-tuple ${\mathbf{d}}_\mcal{I}$ of elements in $\Lambda_U$ indexed by the elements of $\mcal{I}$.
\end{itemize}
\end{definition}

We are particularly interested in the seeds $(\mcal{I}, \mathbf{d})$ of the form
\begin{itemize}
\item $\mcal{I} := \mcal{I}_{\textup{seed}}$ in~\eqref{IndexsetforSeed}.
\item ${\mathbf{d}}_{\mcal{I}}$ is a tuple of \emph{non-zero real} numbers.
\end{itemize}
Let $\mathbf{y}_\mcal{I}$ denote the components of $\mathbf{y}$ associated with the set $\mcal{I}$ of indices. Namely,
\begin{equation}\label{seeddatay}
\mathbf{y}_\mcal{I} := \left( y_{m,m}, y_{i_1,j_1}, \cdots, y_{i_{n-m-1},j_{n-m-1}} \right).
\end{equation}
Then, as the initial step, we take
$\mathbf{y}_\mcal{I} := {\mathbf{d}}_{\mcal{I}}.$
So, the double indices designate the places in which the components of ${\mathbf{d}}_{\mcal{I}}$ are plugged. Since $\mcal{I}$ is always taken to be $\mcal{I}_{\textup{seed}}$, $\mcal{I}$ will be often omitted from now on.
We instead set $d_{i,j}$ to denote the component of ${\mathbf{d}}_{\mcal{I}}$ corresponding to $(i,j)$.

We shall determine a solution of the SLT-equation recursively from a seed according to the order
$\prec_{\mathrm{hor}}$ (or $\prec_{\mathrm{ver}}$) and the diagonal filtration~\eqref{equ_nestseq}.
By the SLT-equation and Lemma~\ref{extensionlemma} over the field of complex numbers,
$\mathbf{y}_{(r)}$, $y_{m,m}$ and one datum of a seed~\eqref{seeddatay} in the  diagonal $i+j = r+1$ in $\mathbf{y}_{(r+1)} \backslash \, \mathbf{y}_{(r)}$ generate the other variables in the diagonal with a suitable choice of complex numbers
$$
\mathbf{c} := \left( c^\textup{hor}_{k, k+1}, \cdots, c^\textup{hor}_{n-1, n} , c^\textup{ver}_{k+1, k}, \cdots, c^\textup{ver}_{n, n-1} \right).
$$
Namely, we isolate and express one undetermined variable in terms of the already determined variables for the remaining $y_{i,j}$'s and $\mathbf{c}$ inductively over the order $\prec_{\mathrm{hor}}$ (or $\prec_{\mathrm{ver}}$) by rewriting the SLT-equation.
However, the undetermined variable might be \emph{zero} or \emph{undefined} which should be
avoided. A \emph{good} choice of seed, we call a \emph{generic} seed,
will do this purpose.

Now we provide the precise meaning of what we mean by a generic seed.
In the setup of~\eqref{settingup}, we rewrite $\ell^{\frak{b},m}_{(i,j)}  ({\mathbf{y}}) =0$ into~\eqref{tildepabijy}.
We first note that \eqref{tildepabijy} involves only the variables
 $y_{(r,s)}$ such that $(r,s) \prec_{\mathrm{hor}} (i+1,j)$ (resp. $(r,s) \prec_{\mathrm{ver}} (i,j+1)$)
 for $i \geq j$ (resp. for $i < j$). Therefore we can determine $y_{(i+1,j)}$ (resp. $y_{(i,j+1)}$)
 if $i \geq j$ (resp. if $i < j$), \emph{as long as the value of $h^{\frak{b},m}_{(i,j)}  ({\mathbf{y}}) \neq 0$
 or is undefined, i.e., provided}
 \begin{equation}\label{eq:welldefined}
{\mathbf{y}} \not \in {\mathrm{Zero }}(h^{\frak{b},m}_{(i,j)}) \cup {\mathrm{Pole }}(h^{\frak{b},m}_{(i,j)}).
\end{equation}
\begin{definition}
A seed $(\mcal{I},\mathbf{{d}}_\mcal{I})$ is called \emph{generic}
if the seed $\mathbf{y}_{\mcal{I}} =  {\mathbf{d}}_{\mcal{I}}$
generates inductive solutions $y_{(r,s)}$ for $(r,s) \prec_{\mathrm{hor}} (i+1,j)$
(resp. $(r,s) \prec_{\mathrm{ver}} (i,j+1)$) such that
\begin{equation}\label{widetildepabijneq}
h^{\frak{b},m}_{(i,j)}  ({\mathbf{y}}) \neq  0 \mod T^{>0}
\end{equation}
for all $(i,j)$ with $i \geq j$ (resp. for all $(i,j)$ with $i< j$).
\end{definition}
\begin{example}
A straightforward calculation asserts that the tuples
\begin{enumerate}
\item $\mcal{I} = \mcal{I}_\textup{seed} = ( (2,2), (1,3), (2,3), (3,3), (3,4) )$
\item ${\mathbf{d}}_{\mcal{I}} = (-1, 1, 1, -1, 1)$
\end{enumerate}
form a generic seed of $\Gamma(7)$ to $B(2)$. The tuples
\begin{enumerate}
\item $\mcal{I} = \mcal{I}_\textup{seed} = ( (2,2), (1,3), (2,3), (3,3), (3,4) )$
\item ${\mathbf{d}}_{\mcal{I}} = (-1, 1, 1, 1, 1)$
\end{enumerate}
form a seed of $\Gamma(7)$ to $B(2)$, but not a generic seed because $h^{\frak{b},m=2}_{(1,5)}  ({\mathbf{y}}) = 0$.
\end{example}

The main proposition of this section is the following existence of a generic seed,
whose proof will be occupied by the rest of this section.
\begin{proposition}\label{Existenceofgenericseeds}
For each $m \in \Z$ where $2 \leq m \leq k =  \left\lceil n/2 \right\rceil$, a generic seed of $\Gamma(n)$
associated to $B(m)$ exists.
\end{proposition}

As a corollary, we assert solvability of the SLT-equation.
\begin{corollary}\label{splitleadingtermequationissolv}
The SLT-equation of $\Gamma(n)$ associated with $B(m)$ has a solution each component of which is a non-zero complex number.
\end{corollary}

\begin{proof}
Assume that a seed has the property~\eqref{widetildepabijneq}, the remaining $y_{i,j}$'s and a sequence $\mathbf{c}$ are (uniquely) determined in $\C^*$ by the exactly same process as in Section~\ref{outsideofbm}.
\end{proof}

\subsection{Pre-generic elements}

We again exploit the coordinate system $\{z_{i,j} \mid (i,j) \in \Gamma(n) \backslash B(m) \cup \{(m,m)\} \}$ in~\eqref{coordinatechange}.
Recall the system of rational functions in~\eqref{tildepabijz}.

We have the following lemma.
\begin{lemma}
$k^{\frak{b},m}_{(i,j)}  ({\mathbf{z}}) \neq 0$ if and only if $h^{\frak{b},m}_{(i,j)}  ({\mathbf{y}}) \neq 0$ for $(i,j) \in \Gamma(n) \backslash B(m)$.
\end{lemma}

\begin{proof}
Under the coordinate change~\eqref{coordinatechange},~\eqref{tildepabijy} is converted into~\eqref{tildepabijz}.
\end{proof}

When $m < k = \lceil n / 2 \rceil$,
since $\mcal{I} \subset B(k)$ and $\mathbf{y}_\mcal{I} = \mathbf{z}_\mcal{I}$ by~\eqref{coordinatechange}, we may insert ${\mathbf{d}}_{\mcal{I}}$ into $\mathbf{z}_\mcal{I}$ as the starting point. For the case $m = k$, we have taken $c^{\textup{hor},\C}_{m,m+1} = 1$ and $c^{\textup{ver}, \C}_{m+1,m} = 1$ and hence $\mathbf{z}_\mcal{I} = \mathbf{y}_\mcal{I} = {\mathbf{d}}_{\mcal{I}}$ as well.
As in~\eqref{equ_nestseq}, consider the diagonal filtration of the field of rational functions (with respect to $z$)
$$
\C \subset \C ({\mathbf{z}_{(1)}}) \subset \C ({\mathbf{z}_{(2)}}) \subset \dots \subset \C ({\mathbf{z}_{(n)}})
$$
where
$$
{\mathbf{z}_{(r)}} := \bigcup_{s=2}^r \{z_{1,s-1}, \dots, z_{s-1,1} \} \backslash \Big( \mathbf{z}_{B(m)} \backslash \{ z_{m,m} \} \Big) \mbox{ and } \mathbf{z}_{B(m)} = \{ z_{i,j} \mid (i,j) \in B(m) \}.
$$

Choosing one component in a diagonal generically, we can make the first two equations of~\eqref{tildepabijz} non-zero because of the following lemma.
Let $(r,s) \in B(m)\setminus\{(m,m)\}$ be given.
Denote by $X(i)$ the rational function obtained by successive composition of
$k^{\frak{b},m}_{(j,k)}$ for all $(j,k) \prec_{\mathrm{hor}} (r,s)$. If the above mentioned
inductive determination does not meet an obstruction, its value corresponding at ${\bf y}$ is supposed to
determine the value of $z_{r-i,s+i}$. We only show the case for $i > 0$ since the case where $i < 0$ similar.

\begin{proposition}\label{rationalfunctionzij}
The rational function $X(i)$ is a non-constant fractional linear map of $z_{(r,s)}$ in coefficients $\C(\mathbf{z}_{(r+s-1)})$.
\end{proposition}

\begin{proof}
By~\eqref{tildepabijz}, we derive a recurrence relation for $X(i)$'s of the form
\begin{equation}\label{recrelxi}
X(i) = [i] + \frac{[i,i-1]}{X(i-1)}
\end{equation}
where
\begin{equation}\label{[i][i,i-1]}
[i] := - z_{r-i-1, s+i} + \frac{(z_{r-i, s+i-1})^2}{z_{r-i,s+i-2}},\,\, [i, i-1] :=(z_{r-i, s+i-1})^2.
\end{equation}
We note that $[i], \, [i,i-1] \in \C(\mathbf{{z}}_{(r+s-1)})$, while
$X(i) \in \C({\mathbf{z}}_{(r+s)})$ for all $i$.

Composing~\eqref{recrelxi} several times, $X(i)$ is expressed as a continued fraction in terms of $X(0) = z_{r,s}$.
If we set the initial condition $A(0) = 1, \, B(0) = 0$, we can express $X(i)$ as
\begin{equation}\label{X(i)}
X(i) = \frac{A(i) \cdot X(0) + B(i)}{A(i-1) \cdot X(0) + B(i-1)}
\end{equation}
for the elements  $A(i), \, B(i) \in \C(\mathbf{{z}}_{(r+s-1)})$.
Thus, we have $X(i) \in \C ({\mathbf{z}}_{(r+s-1)})(z_{r,s})$.

To show that every $X(i)$ is \emph{non-constant} with respect to $X(0)=z_{(r,s)}$,
we investigate properties of $A(i)$'s and $B(i)$'s.
By the induction, we obtain the following lemma.
\begin{lemma}
Consider $\{ i, i-1, \cdots, 1, 0 \}$.
Let $\mcal{P}_{A(i)}$ be the partitions of $\{ i, i-1, \cdots, 1 \}$ into one single number or two consecutive numbers and $\mcal{P}_{B(i)}$ the partitions of $\{ i, i-1, \cdots, 1, 0 \}$ into one single or two consecutive numbers containing the subset $[1,0]$. Then
$$
A(i) = \sum_{I \in \mcal{P}_{A(i)}} [I] \mbox{ and } B(i) = \sum_{I \in \mcal{P}_{B(i)}} [I],
$$
\end{lemma}
For instance, $A(3)$ and $B(3)$ are expressed as
\begin{align*}
A(3) &= [3][2][1] + [3][2,1] + [3,2][1], \\
B(3) & = [3][2][1,0] + [3,2][1,0].
\end{align*}

\begin{corollary}
\begin{align*}
&A(i) = [i] \cdot A(i-1) + [i, i-1] \cdot A(i-2) \\
&B(i) = [i] \cdot B(i-1) + [i, i-1] \cdot B(i-2).
\end{align*}
\end{corollary}

Note that $X(0)$ and $X(1)$ are non-constant Laurent polynomials with respect to $X(0) = z_{r,s}$. Suppose to the contrary that $X(i)$ is a constant function with the value $C$ and all $X(j)$'s for all $j < i$ are non-constant rational functions of $X(0) (=z_{r,s})$.

Suppose $X(i) \equiv C$.
By recalling $A(i), \, B(i) \in \C(\mathbf{{z}}_{(r+s-1)})$ and
substituting the values $z_{r,s} = 0, \, 1$ into~\eqref{X(i)} respectively, we obtain
\begin{align*}
&C \cdot A(i-1) = A(i) = [i] \cdot A(i-1) + [i, i-1] \cdot A(i-2)\\
&C \cdot B(i-1) = B(i) = [i] \cdot B(i-1) + [i, i-1] \cdot B(i-2).
\end{align*}
Then $(C-[i]) \cdot A(i-1) = [i,i-1] \cdot A(i-2)$.

We claim that $C - [i]$ is a non-zero rational function contained in $\C({\mathbf{z}}_{(r+s-1)})$. Otherwise
we would have $A(i-2) = B(i-2) = 0$ because $[i, i-1] = (z_{r-i, s+i-1})^2 \neq 0$. It would yield that $X(i-2) \equiv 0$, contradicting to the assumption that $X(i-2)$ is not constant with respect to $z_{r,s}$. We then have
\begin{align*}
A(i-1) = C^\prime \cdot A(i-2) \\
B(i-1) = C^\prime \cdot B(i-2)
\end{align*}
where $C^\prime = {[i,i-1]}/{(C - [i])} \in \C({\mathbf{z}}_{(r+s-1)})$. Then we deduce $X(i-1) = C^\prime$ which is constant with respect to $z_{r,s}$, a contradiction to the standing hypothesis. This finishes
the proof of the proposition.
\end{proof}

\begin{corollary}\label{choiceofzij}
Suppose that the set $\mathbf{{z}}_{(r+s-1)}$ is determined. There exists a non-zero real number $d_{r,s}$ such that if we set $z_{r,s} = d_{r,s}$, then we achieve
\begin{equation}\label{paijznonzero}
k^{\frak{b},m}_{(i,j)}  ({\mathbf{z}}) \neq 0
\end{equation}
for all $(i,j)$'s satisfying $i + j = r + s - 1$.
\end{corollary}

\begin{proof}
Since each ${X(i)}$ is a non-constant fractional
linear map of $z_{r,s}$, there are only finitely many $z_{r,s}$'s making $z_{r-i, s+i}$ zero or not defined.
Avoid these values when choosing $d_{r,s}$.
\end{proof}

\begin{definition}
Suppose the set ${\mathbf{z}}_{(r+s-1)}$ is given.
For an index $(r,s) \in \Gamma(n) \backslash B(m)$, an element $d_{r,s}$ in a seed is said to be \emph{pre-generic with respect to}
${\mathbf{z}}_{(r+s-1)}$ if~\eqref{paijznonzero} holds for any $(i,j) \in \Gamma(r+s) \backslash (B(m) \cup \Gamma(r+s-1))$.
\end{definition}

For the later purpose, we prove the following property of the the pre-generic elements.

\begin{lemma}\label{symmericseedz} Let $m < k$. Assume that $\mathbf{{z}}_{(s+m)}$
are previously determined. Suppose that we are given a seed
$$
\{d_{m,m}\} \cup \{d_{s, m+1} \in \C^* \mid 1 \leq s \leq m\}
$$
such that for each $s$ $d_{s,m+1}$ is pre-generic with respect to a choice of
$d_{m,m}$ and $d_{i,m+1}$ with $1 \leq i \leq s-1$.
Then, regardless of a choice of non-zero $d_{m,m}$, the same $d_{s,m+1}$
remains to be pre-generic as long as we do not change $d_{1,m+1}, \cdots, d_{s-1,m+1}$.
\end{lemma}

\begin{proof}
If $m < k$, we see $y_{i,j} = z_{i,j}$ for $(i,j) \in B(m)$ by~\eqref{coordinatechange}.
We will prove
\begin{equation}\label{eq:zji+1}
z_{j,i+1} = (-1)^{i+j} \cdot \frac{(d_{m,m})^2}{z_{i+1,j}}
\end{equation}
by the induction over the order $\prec_{ver}$.
Recall from~\eqref{pamly} that
\begin{equation*}
z_{m+1,i} = (-1)^{i + (m+1) - 1} \cdot \frac{(d_{m,m})^2}{z_{i,m+1}},
\end{equation*}
which provides the initial step for the induction. Next suppose that \eqref{eq:zji+1}
holds for all $(r,s) \prec_{ver} (j,i+1)$. Then we observe that all 4 sub-indices are smaller than $(j,i+1)$
appearing in the equation
\begin{equation}\label{eq:4term-equation}
- \frac{z_{i,j+1}}{z_{i,j}} - \frac{z_{i-1,j}}{z_{i,j}} + \frac{z_{i,j}}{z_{i+1,j}} + \frac{z_{i,j}}{z_{i,j-1}} = 0.
\end{equation}
By the induction hypothesis, we substitute \eqref{eq:zji+1} thereinto, we derive
$$
0 = k^{\frak{b},m}_{i,j}({\bf z}) = \frac{z_{j,i}}{z_{j+1,i}} + \frac{z_{j,i}}{z_{j,i-1}} + (-1)^{i + j -1}\frac{(d_{m,m})^2}{z_{i+1,j} \cdot z_{j,i}}
- \frac{z_{j-1,i}}{z_{j,i}}.
$$
Then by a back-substitution of \eqref{eq:4term-equation} thereinto, we obtain
$$
\frac{z_{j,i+1}}{z_{j,i}} + (-1)^{i + j- 1} \frac{(d_{m,m})^2}{z_{i+1,j}\, z_{j,i}} = 0
$$
which is equivalent to
\begin{equation*}
z_{j,i+1} = (-1)^{i+j} \cdot \frac{(d_{m,m})^2}{z_{i+1,j}}.
\end{equation*}
Therefore, $k^{\frak{b},m}_{(i,j)}  ({\mathbf{z}})  \neq 0$ if and only if
$k^{\frak{b},m}_{(j,i)}  ({\mathbf{z}})  \neq 0$. Hence, Lemma~\ref{symmericseedz} follows.
\end{proof}

\subsection{Generic seeds}

Applying Corollary~\ref{choiceofzij}, we make the first two expressions in~\eqref{tildepabijz} non-zero by generically choosing the value of one entry of each diagonal.
To make the last two terms of ~\eqref{tildepabijz} non-zero
at the same time, we need to carefully adjust this choice we have made. We recall
$k: = \lceil n/2 \rceil$.
Depending on the parity of $n$, either $n = 2k-1$ or $n=2k$. We consider the two cases separately.

\smallskip

\noindent{\bf Case 1. $ n = 2k - 1$.} We need several lemmas.

\begin{lemma}\label{reductionkk-1k-1k-1}
Assume that $\mathbf{{z}}_{(n-2)}$ is given. Suppose that either $d_{k-1,k-1} = -1$ is pre-generic or $k - 1 =m$. Then there is a real number $d_{k-1,k-1}$ (sufficiently close to $-1$, but not equal to $- 1$) and a non-zero real number $d_{k-1,k}$ such that if $z_{k-1,k-1} = d_{k-1,k-1}$ and $z_{k-1,k} = d_{k-1,k}$,
\begin{equation}\label{pabijzneq0}
k^{\frak{b},m}_{(i,j)}  ({\mathbf{z}}) \neq 0 \mod T^{>0}
\end{equation}
for all $(i,j)$ with $i + j = n-1$ and $i + j = n$.
\end{lemma}

Note that $k^{\frak{b},m}_{(i,j)}  ({\mathbf{z}})$'s for $(i,j)$ with $i + j = n$ provide the last two expressions of~\eqref{tildepabijz}.

\begin{proof}
Assuming that $d_{k-1, k-1} = -1$ is pre-generic, by definition, every $z_{i, n - 1- i}$ is defined and becomes non-zero if we set $z_{k-1,k-1} = d_{k-1,k-1} = -1$. By Corollary~\ref{choiceofzij}, we can choose and fix a pre-generic element $d_{k-1,k}$ for $z_{k-1,k}$ so that the entries $z_{i,n-i}$'s are also determined.

We emphasize that $d_{k-1,k-1} = -1$ is \emph{never} being a component of a generic seed because of the following reason. It is straightforward to see that the equations $\ell^{\frak{b},m}_{(i,j)}  ({\mathbf{y}}) = 0$ in terms of ${\mathbf{z}}$ for $(i,j)$'s with $i + j \geq n$ and $i < j$ read
\begin{equation}\label{lemma1012recall}
\begin{cases}
\prod^{k}_{r=k} c^\textup{ver}_{r+1,r}&= - {z_{k-2,k}} + {(z_{k-1,k})^2} \left( 1 + \frac{1}{z_{k-1,k-1}} \right), \\
\prod^{k+1}_{r=k} c^\textup{ver}_{r+1,r} &= - {z_{k-3,k+1}} + {(z_{k-2,k+1})^2} \left( \left( \prod^{k}_{r=k} c^\textup{ver}_{r+1,r}\right)^{-1} + \frac{1}{z_{k-2,k}} \right), \\
&\cdots \\
\prod^{n-2}_{r=k} c^\textup{ver}_{r+1,r} &= - {z_{1,n-2}} + {(z_{2,n-2})^2} \left( \left( \prod^{n-3}_{r=k} c^\textup{ver}_{r+1,r}\right)^{-1} + \frac{1}{z_{2,n-3}} \right), \\
\prod^{n-1}_{r=k} c^\textup{ver}_{r+1,r} &= - {(z_{1,n-1})^2} \left( \left( \prod^{n-2}_{r=k} c^\textup{ver}_{r+1,r}\right)^{-1} + \frac{1}{z_{1,n-2}} \right).
\end{cases}
\end{equation}
If one chooses $z_{k-1,k-1} = d_{k-1, k-1} = -1$, then we obtain
\begin{align*}
\prod^{j}_{r=k} c^\textup{ver}_{r+1,r} = - {z_{n-j-1,j}}
\end{align*}
from~\eqref{lemma1012recall} for $j = k, k+1, \cdots, n-2$ and
\begin{equation}\label{problematicpamb}
k^{\frak{b},m}_{(1,n-1)}  ({\mathbf{z}})
  = \prod^{n-1}_{r=k} c^\textup{ver}_{r+1,r} = 0.
\end{equation}
Thus, the seed ${\mathbf{d}}_{\mcal{I}}$ is \emph{not} generic.
Nevertheless, we claim that there exists a choice of $d_{k-1, k-1}$ not equal to $-1$ but close to $-1$ so that~\eqref{pabijzneq0} is satisfied.

Note that the fixed $d_{k,k-1}$ remains to be pre-generic even if we perturb the value $z_{k-1, k-1}$ from $-1$ with sufficiently small amount. This is because the expression $k^{\frak{b},m}_{(i,j)}  ({\mathbf{z}}) $ for each index $(i, j)$ with $i + j = n - 1$ is a continuous function with respect to $z_{k-1, k-1}$ at $-1$ after inserting $d_{k, k-1}$ into $z_{k, k-1}$. Also, by Proposition~\ref{rationalfunctionzij}, there exists a dense set of pre-generic elements for $d_{k-1, k-1}$. Therefore,~\eqref{pabijzneq0} is satisfied for $i + j = n -1$.

Also, we observe that as $z_{k-1, k-1} \to -1$, because of~\eqref{tildepabijz} and~\eqref{lemma1012recall}, $k^{\frak{b},m}_{(n-j,j)}  ({\mathbf{z}}) \to - z_{n-j-1,j}$ when $j \geq k$. Because $-z_{n -j-1, j} \neq 0$ for $j$ with $k \leq j < n - 1$, we still have $k^{\frak{b},m}_{(n-j,j)}  ({\mathbf{z}}) \neq 0$ for $j$ with $k \leq j < n - 1$ if $d_{k-1,k-1}$ is sufficiently close to $-1$. Finally, we claim that $k^{\frak{b},m}_{(1,n-1)}  ({\mathbf{z}}) \neq 0$ if $z_{k-1, k-1} \neq -1$ so that the issue in~\eqref{problematicpamb} is solved.
From~\eqref{lemma1012recall} and $z_{k-1, k-1} \neq -1$, it follows
$$
c^\textup{ver}_{k+1, k} \neq -z_{k-2, k}.
$$
Proceeding inductively, we deduce $k^{\frak{b},m}_{(1,n-1)}  ({\mathbf{z}}) \neq 0$ from~\eqref{lemma1012recall}. The discussion on the part where $j < k$ is omitted because the argument is symmetrical.
Consequently, we may choose a generic $d_{k-1, k-1}$ sufficiently close to $-1$ so that~\eqref{pabijzneq0} holds for all $(i,j)$ with $i + j = n-1$ and $i + j = n$.

It remains to take care of the case where $k-1 = m$. $(k-1,k-1) = (m,m)$ is contained in the box $B(m)$ so that $d_{k-1,k-1}$ can be freely chosen by Corollary~\ref{corsymmetricsol}. Thus, the exactly same argument can be applied as above.
\end{proof}

By applying a similar argument, we can prove the following lemma.

\begin{lemma}\label{reductioniii+1i+1}
Suppose that either $d_{i, i} = \pm 1$ is pre-generic for $i > m$ or $i = m$. There is a real number $d_{i, i}$ (sufficiently close to $\pm 1$, but not equal to $\pm 1$) and a non-zero real number $d_{i,i+1}$ so that $d_{i+1, i+1} = \mp 1$ becomes pre-generic.
\end{lemma}

We now ready to start the proof of Proposition~\ref{Existenceofgenericseeds} for the case where $n = 2k - 1$ and $m < k := \left\lceil n/2 \right\rceil$.

\begin{proof}[Proof of Proposition~\ref{Existenceofgenericseeds}]
We start with a tentative choice of $d_{m,m} = \pm 1$. Choosing pre-generic elements from $d_{1,m+1} := z_{1,m+1}$ to $d_{m-1,m+1} := z_{m-1, m+1}$, we find $\mathbf{{z}}_{(2m)}$ so that \eqref{pabijzneq0} is satisfied for each index $(i,j)$ with $i + j \leq 2m - 1$. Due to Lemma~\ref{reductioniii+1i+1}, we may select $d_{m,m}$ sufficiently close to $\pm 1$ and $d_{m,m+1}$ so that $d_{m+1, m+1} = \mp 1$ becomes pre-generic. Because of Lemma~\ref{symmericseedz}, note that $d_{1,m+1}, \cdots, d_{m, m+1}$ remain to be pre-generic even if we choose another $d_{m,m}$.
Moreover, applying Lemma~\ref{reductioniii+1i+1} repeatedly, we assert that $d_{k-1,k-1} = -1$ is also pre-generic by suitably choosing $d_{\bullet, \bullet}$. Hence, we have \eqref{pabijzneq0} for all indices $(i,j)$'s with $i + j \leq n -2$. Finally, Lemma~\ref{reductionkk-1k-1k-1} says that there is $d_{k-1, k-1}$ and $d_{k,k-1}$ such that~\eqref{pabijzneq0} holds for $i + j = n-1, n$. Thus, we have just found a generic seed.
\end{proof}

\smallskip

\noindent{\bf Case 2. $ n = 2k$ and $m < k$.} Modifying the proofs of Lemma~\ref{reductionkk-1k-1k-1} and Lemma~\ref{reductioniii+1i+1}, we can prove the following.

\begin{lemma}\label{evenreduction1}
Assume that $\mathbf{{z}}_{(n-2)}$  is given. Suppose that $d_{k-1,k} = -1$ is pre-generic. Then there are a real number $d_{k-1,k}$ (sufficiently close to $-1$, but not equal to $-1$) and a non-zero real number $d_{k,k}$ such that if $z_{k-1,k} = d_{k-1,k}$ and $z_{k,k} = d_{k,k}$,
$$
k^{\frak{b},m}_{(i,j)}  ({\mathbf{z}}) \neq 0 \mod T^{>0}
$$
for all $(i,j)$ with $i + j = n-1$ and $i + j = n$.

Suppose that $d_{i-1, i} = \pm 1$ is pre-generic for $i \geq m+1$. There are a real number $d_{i-1, i}$ (sufficiently close to $\pm 1$) and a non-zero real number $d_{i,i}$ so that $d_{i, i+1} = \mp 1$ becomes pre-generic.
\end{lemma}

Also, we need the lemma, which serves as the starting point to obtain the desired $d_{\bullet,\bullet}$'s.

\begin{lemma}\label{dm+11canbegeneric}
$d_{m, m+1} = \pm 1$ can be pre-generic.
\end{lemma}

\begin{proof}
As in Lemma~\ref{symmetricsol}, one sees that
\begin{equation}
\displaystyle \widetilde{z}_{i,m+j} :=
\begin{cases}
1 \, &\mbox{for } i = j \\
\displaystyle \prod_{r=0}^{j - i - 1} (2i + 2r) \, &\mbox{for } i < j \\
\displaystyle \prod_{r=0}^{i - j - 1} (2j + 2r)^{-1} \, &\mbox{for } i > j
\end{cases}
\end{equation}
is a solution of $\ell^{\frak{b},m}_{(i,j)}  ({\mathbf{y}})  = 0$ under the coordinate change~\eqref{coordinatechange} for $m + i + j < n$. Also, by Lemma~\ref{minussol}, so is
\begin{equation}\label{solutionofpamij2}
z_{i,m+j} := a \cdot \widetilde{z}_{i,m+j}
\end{equation}
for any non-zero complex number $a$. Selecting
$$
a := \prod^{m-2}_{r=0} (2 + 2r),
$$
$d_{m, m+1} = {z}_{m,m+1}$ becomes $1$. Because of Lemma~\ref{symmericseedz}, no matter what we choose any non-zero complex number $d_{m,m}$, $d_{m,m+1}$ is pre-generic (with respect to the previous determined $\mathbf{z}_{(2m)}$).
\end{proof}

\begin{proof}[Proof of Proposition~\ref{Existenceofgenericseeds} (continued)]
Combining Lemma~\ref{evenreduction1} and Lemma~\ref{dm+11canbegeneric}, we conclude Proposition~\ref{Existenceofgenericseeds} for the case where $n = 2k$ and $m < k$.
\end{proof}

\smallskip

\noindent{\bf Case 3. $ n = 2k$ and $m = k$.} In this case, we have taken $d_{m,m} = 1$. Because of Lemma~\ref{symmetricsol}, note that
\begin{equation}
\displaystyle \widetilde{z}_{i,m+j} :=
\begin{cases}
1 \, &\mbox{for } i = j \\
\displaystyle \prod_{r=0}^{j -i - 1} (2i + 2r) \, &\mbox{for } i < j \\
\displaystyle \prod_{r=0}^{i - j - 1} (2j + 2r)^{-1} \, &\mbox{for } i > j
\end{cases}, \,\,
\displaystyle \widetilde{z}_{m+i,j} :=
\begin{cases}
\displaystyle {(-1)^{m + i + j -1}} \, &\mbox{for } i = j \\
\displaystyle (-1)^{m + i + j -1} \prod_{r=0}^{i - j - 1} (2j + 2r)^{-1} \, &\mbox{for } i > j \\
\displaystyle (-1)^{m + i + j -1} \prod_{r=0}^{j - i - 1} (2i + 2r) \, &\mbox{for } i < j
\end{cases}
\end{equation}
respectively form a solution of $\ell^{\frak{b},m}_{(i,m+j)}  ({\mathbf{y}})  = 0$ and $\ell^{\frak{b},m}_{(m+i,j)}  ({\mathbf{y}})  = 0$ for $m + i + j < n$. Also, our choice makes $\ell^m_{(l)} (\mathbf{{y}}) = 0$ in~\eqref{pamly} because $c^\textup{ver}_{m+1,m} = 1$ and $c^\textup{hor}_{m,m+1} = 1$.

Furthermore,
\begin{equation}\label{solutionofpamij2}
z_{i,m+j} := a \cdot \widetilde{z}_{i,m+j}, \quad z_{m+i,j} := a^{-1} \cdot \widetilde{z}_{m+i,j}
\end{equation}
are also solutions for any non-zero complex number $a$. Thus, we have a one-parameter family of solutions. Then the expressions $k^{\frak{b},m}_{(i,m+j)}(\mathbf{{z}})$ and $k^{\frak{b},m}_{(m+i,j)}(\mathbf{{z}})$ with $i + j = n$ can be considered as a function with respect to $a$.

Because of Lemma~\ref{evenreduction1}, it suffices to deal with the starting point.

\begin{lemma}\label{tempssedmk}
There exists a choice of the variable $a$ such that
\begin{equation}\label{genericamim-i}
k^{\frak{b},m}_{(m-i,m+i)}(\mathbf{{z}}) \neq 0 \, \textup{ and } \, k^{\frak{b},m}_{(m+i,m-i)}(\mathbf{{z}}) \neq 0 \mod T^{>0}
\end{equation}
in~\eqref{tildepabijz}.
\end{lemma}

\begin{proof}
We claim that $k^{\frak{b},m}_{(m-i,m+i)}(\mathbf{{z}}) / z_{m-i, m+i}$ is a non-constant rational function w.r.t. $a$. For $i \geq 1$, we observe
$$
\frac{k^{\frak{b},m}_{(m-i,m+i)}(\mathbf{{z}})}{z_{m-1, m+1}} = - \frac{z_{m-2,m+1}}{z_{m-1,m+1}} + z_{m-1,m+1} =  - (2m - 4) + a \cdot \left( \prod^{m-3}_{r = 0} \left( 2 + 2r \right)^{-1} \right)
$$
is a non-constant linear function with respect to $a$. By the induction hypothesis, assume that
$$
\frac{k^{\frak{b},m}_{(m-i,m+i)}(\mathbf{{z}})}{z_{m-i, m+i}} := \frac{P_i (a)}{Q_i (a)}
$$
is a non-constant rational function with respect to $a$. Then, we see
\begin{align*}
\frac{k^{\frak{b},m}_{(m-i-1,m+i+1)}(\mathbf{{z}})}{z_{m-i-1, m+i+1}} &= \left( - \frac{z_{m-i-2,m+i+1}}{z_{m-i-1,m+i+1}} + \frac{z_{m-i-1,m+i+1}}{z_{m-i-1,m+i}} \right) + \frac{z_{m-i-1,m+i+1}}{k^{\frak{b},m}_{(m-i,m+i)}(\mathbf{{z}})} \\
&= \left( - \frac{\widetilde{z}_{m-i-2,m+i+1}}{\widetilde{z}_{m-i-1,m+i+1}} + \frac{\widetilde{z}_{m-i-1,m+i+1}}{\widetilde{z}_{m-i-1,m+i}} \right) + \frac{\widetilde{z}_{m-i-1,m+i+1}}{\widetilde{z}_{m-i,m+i}} \cdot \frac{Q_i(a)}{P_i(a)},
\end{align*}
which is also a non-constant rational function. Similarly, one can see that $k^{\frak{b},m}_{(m-i,m+i)}(\mathbf{{z}})$ is also a non-constant rational function for $i \geq 1$. Thus,~\eqref{genericamim-i} is established if choosing $a$ generically.
\end{proof}

We are ready to prove Proposition~\ref{Existenceofgenericseeds} for the case where $n = 2k$ and $m = k :=  \left\lceil n/2 \right\rceil$.

\begin{proof}[Proof of Proposition~\ref{Existenceofgenericseeds} (continued)]
By Lemma~\ref{tempssedmk}, we choose $d_{i, m+1} := a \cdot \widetilde{z}_{i,m+1}$ from~\eqref{solutionofpamij2} as a generic seed. It complete the proof.
\end{proof}

\subsection{Proof of Theorem~\ref{theoremD}}

Finally, we are ready to prove the following main theorem.

\begin{theorem}[Theorem \ref{theoremD}]\label{theorem_completeflagmancotinuum}
Let $\lambda = \{ \lambda_{i} := n - 2i + 1 \,|\, i = 1, \cdots, n \}$ be an $n$-tuple of real numbers for an arbitrary integer $n \geq 4$.
Consider the co-adjoint orbit $\mathcal{O}_\lambda$, a complete flag manifold $\mathrm{Fl}(n)$ equipped with the monotone Kirillov--Kostant--Souriau symplectic form $\omega_\lambda$.
Then each Gelfand--Cetlin fiber $L_m(t)$ is non-displaceable Lagrangian for every $2 \leq m \leq \left\lfloor \frac{n}{2} \right\rfloor$.

In particular, there exists a family of non-displaceable non-torus Lagrangian fibers
$$\left\{L_m(1) \mid 2 \leq m \leq \left\lfloor \frac{n}{2} \right\rfloor, m \in \Z \right\}$$
where $L_m(1)$ is diffeomorphic to $\mathrm{U}(m) \times T^{\frac{n(n-1)}{2} - m^2}$.
\end{theorem}

\begin{proof}
By Corollary~\ref{splitleadingtermequationissolv}, the SLT-equation~\eqref{splitleadingtermequ} has a desired solution for some nonzero complex numbers $c^{\textup{ver},\C}_{i+1, i}$'s and $c^{\textup{hor}, \C}_{j, j+1}$'s ($i, j \geq k$). Theorem~\ref{splitleadingtermequationimpliesnonzero} ensures that for each Lagrangian torus $L_m(t)$ ($0 \leq t < 1$), there exists a suitable bulk-deformation parameter $\frak{b}$ of the form~\eqref{bulkparak} so that $W^\frak{b}$ admits a critical point.  By Theorem~\ref{criticalpointimpliesnondisplaceability}, each GC torus fiber $L_m(t)$ for $0 \leq t < 1$ is non-displaceable. Furthermore, Lemma~\ref{closednessofnondisp} imply that $L_m(1)$ is non-displaceable.
\end{proof}

\section{Calculation of potential function deformed by Schubert cycles}\label{sec_bulkdefbyschucy}

Since the main steps of the derivation of \eqref{bulkdeformedpotential} are the
same as that of the proof of \cite[Proposition 4.7]{FOOOToric2} given in Section 7 therein, we will only
explain modifications we need to make to apply them to the current GC case.
Also for the purpose of proving the counter part of Theorem \ref{Foootoric2potenti} in the present paper, the facts that
a Fano manifold $X_\epsilon$ has a toric degeneration and that we have only to consider codimension two cycles also help us to simplify the study of holomorphic discs contributing to the potential functions.
We closely follow \cite[Section 9]{NNU}.

Let $L$ be a Lagrangian submanifold in a symplectic manifold $X$. Let $\mcal{M}_{k+1; \ell}(X, L; \beta)$ denote the moduli space of stable maps in the class $\beta \in \pi_2(X, L)$ from a bordered Riemann surface $\Sigma$ of genus zero with $(k+1)$ marked points $\{z_s\}_{s=0}^{k}$ on the boundary $\pa \Sigma$ respecting the counter-clockwise orientation and $\ell$ marked points $\{z^+_r \}_{r = 1}^{\ell}$ at the interior of $\Sigma$.
It naturally comes with two types of evaluation maps, at the $i$-th boundary marked point
\begin{equation}\label{evaluationatboundarymarked}
\ev_i \colon \mathcal M_{k+1;\ell}(X, L; \beta) \to L; \quad \ev_i \left( [ \varphi \colon \Sigma \to X, \{z_s\}_{s=0}^{k+1}, \{z^+_r \}_{r = 1}^{\ell}] \right) = \varphi(z_i)
\end{equation}
and at the $j$-th interior marked point
\begin{equation}\label{evaluationatinteriormarked}
\ev_j^{\text{\rm int}} \colon \mathcal M_{k+1;\ell}(X, L; \beta) \to X; \quad \ev_j^{\text{\rm int}} \left( [ \varphi \colon \Sigma \to X, \{z_s\}_{s=0}^{k+1}, \{z^+_r \}_{r = 1}^{\ell}] \right) = \varphi(z^+_j).
\end{equation}
Set $\mcal{M}_{k+1}(X, L; \beta) := \mcal{M}_{k+1; \ell = 0}(X, L; \beta)$, the moduli space without interior marked points and let
$$
\bev_+ := (\ev_1, \cdots, \ev_k).
$$

Recall that an $A_\infty$-structure with the operators
$$
 \frak m_k = \sum_\beta \frak m_{k,\beta} \cdot T^{\omega(\beta)/  2\pi}, \quad  \frak m_{k,\beta}(b_1, \cdots, b_k): = (\ev_0)!(\bev_+)^*(\pi_1^*b_1 \otimes \cdots \otimes \pi_k^*b_k)
$$
on the de Rham complex $\Omega(L)$ is defined via a smooth correspondence.
\begin{equation}\label{smoothcorrespondence}
	\xymatrix{
                              & {\mcal{M}_{k+1}(X, L; \beta)} \ar[dl]_{\bev_+} \ar[dr]^{\ev_0} &
      \\
 L^k & & L}
\end{equation}
where $\pi_i \colon L^k \to L$ denotes the projection to the $i$-th copy of $L$.
For a general symplectic manifold, one should choose a system of compatible Kuranishi structures and CF-perturbations on $\mcal{M}_{k+1,l}(X, L; \beta)$'s in order to apply the above smooth correspondence, see \cite{Fuk, FOOO:Kura1, FOOO:Kura2} for the details of construction. For a Lagrangian toric fiber $L$ in a $2n$-dimensional toric manifold, by constructing a system of compatible $T^n$-equivariant Kuranishi structures and multi-sections, the smooth correspondence can be applied without adding an auxiliary space for perturbing multi-sections to make it submersive, \cite[Section 12]{FOOOToric2}. It is because $\ev_0$ is  submersive by the $T^n$-equivariance.

We now recall the computation of the potential function
of a torus fiber $L_\varepsilon \subset X_\varepsilon$ in \cite{NNU}.
They were able to exploit
the presence of toric degeneration of $X_\epsilon$ to $X_0$ in their computation
the explanation of which is now in order.
For the study of holomorphic discs in $X_0$ which is not smooth,
they used the following notion in Nishinou--Siebert \cite{NS}.

\begin{definition}[Definition 4.1 in \cite{NS}]\label{toricallytransverse}
A holomorphic curve in a toric variety $X$ is called \emph{torically transverse} if it is disjoint
from all toric strata of codimension greater than one. A stable map
$\varphi\colon \Sigma \to X$ is \emph{torically transverse} if $\varphi(\Sigma) \subset X$ is
torically transverse and $\varphi^{-1}(\operatorname{Int}X) \subset \Sigma$ is dense.
Here, $\operatorname{Int}X$ is the complement of the toric divisors in $X$.
\end{definition}

We denote by
$S_0: = \operatorname{Sing}(X_0)$ the singular locus of $X_0$.
Using the classification result \cite{CO} of holomorphic discs attached to
a Lagrangian toric fiber in a smooth toric manifold and the property of the small resolution,
Nishinou--Nohara--Ueda \cite{NNU} proved the following.

\begin{lemma}[Proposition 9.5 and Lemma 9.15 in \cite{NNU}] \label{discsinX0}
Any holomorphic disc $\varphi \colon (\mathbb{D}^2, \partial \mathbb{D}^2) \to (X_0,L_0)$ can be deformed into a holomorphic
disc with the same boundary condition that is torically transverse. Furthermore
the moduli space ${{\mathcal M}}_1(X_0,L_0;\beta)$ is empty if the Maslov
index of $\beta$ is less than two.
\end{lemma}

\begin{lemma}[Lemma 9.9 in \cite{NNU}]\label{nodiscinW0} There is a small neighborhood $W_0$ of
the singular locus $S_0 \subset X_0$ such that no holomorphic discs of Maslov index two intersect $W_0$.
\end{lemma}

Now let $\phi^\prime_\epsilon \colon X_\epsilon \to X_0$ be a (continuous) extension of the flow
$\phi_\epsilon \colon {X}^\textup{sm}_\epsilon \to {X}^\textup{sm}_0$ given in Theorem~\ref{NNUtoricdeg} (\cite[Section 8]{NNU}).
The following is the key proposition which relates the above mentioned holomorphic discs
in $(X_0,L_0)$ to those of $(X_\epsilon,L_\epsilon)$.

\begin{proposition}[Proposition 9.16 in \cite{NNU}]\label{barM1}
For any $\beta \in \pi_2(X_0,L_0)$ of Maslov index two, there is a positive real numbers $0 < \varepsilon \leq 1$
and a diffeomorphism
$$
\psi \colon {{\mathcal M}}_1(X_0,L_0;\beta) \to {{\mathcal M}}_1 (X_\varepsilon,L_\varepsilon;\beta)
$$
such that the diagram
$$
\xymatrix{
H_*({{\mathcal M}}_1(X_0,L_0;\beta)) \ar[d]_{\psi_*} \ar[r]^{\quad \quad (\ev_0)_*} & H_*(L_0) \ar[d]^{(\phi_\varepsilon)^{-1}_*}\\
H_*({{\mathcal M}}_1(X_\varepsilon,L_\varepsilon;\beta)) \ar[r]^{\quad \quad (\ev_0)_*} & H_*(L_\varepsilon)
}
$$
is commutative.
\end{proposition}

\begin{lemma}[Lemma 9.22 in \cite{NNU}]\label{discsinXepsilon} Let $W_\varepsilon: = (\phi_\varepsilon')^{-1}(W_0)$.
There exists $\varepsilon_0 > 0$ such that
for all $0< \varepsilon \leq \varepsilon_0$, any holomorphic curve bounded by $L_\varepsilon$ in a class
of Maslov index two does not intersect $W_\varepsilon$.
\end{lemma}

We now combine the diffeomorphism $\psi \colon {{\mathcal M}}_1(X_0,L_0;\beta) \to {{\mathcal M}}_1 (X_\varepsilon,L_\varepsilon;\beta)$
and $\phi_\varepsilon' \colon X_0 \to X_\varepsilon$ to define
an isomorphism between the following two correspondences
\begin{equation}\label{smoothcorrespondence0}
	\xymatrix{
                              & {\mcal{M}_{k+1}(X_0, L_0; \beta)} \ar[dl]_{\bev_+} \ar[dr]^{\ev_0} &
      \\
 L_0^k & & L_0} \mbox{ and }
 	\xymatrix{
                              & {\mcal{M}_{k+1}(X_\varepsilon, L_\varepsilon; \beta)} \ar[dl]_{\bev_+} \ar[dr]^{\ev_0} &
      \\
 L_\varepsilon^k & & L_\varepsilon}.
\end{equation}

Although they did not explicitly mention a choice of compatible systems of Kuranishi structures or
perturbations, Nishinou--Nohara--Ueda \cite{NNU} essentially constructed an $A_\infty$-structure
on $L_\varepsilon \subset X_\epsilon$ and computed its potential function
in the same way as on a Fano toric manifolds \cite{CO,FOOOToric1} using Proposition \ref{barM1} and Lemma~\ref{discsinXepsilon}.
Thus, they were able to take advantage of properties of $T^n$-equivariant perturbation in a toric manifold, an open submanifold of a toric variety $X_0$.
We denote the corresponding compatible system of multi-sections by $\frak s = \frak s_{k+1, \beta}$,
see \cite{FOOOToric1, FOOOToric2} for the meaning of this notation.

Next we need to involve bulk deformations for our purpose of a construction of continuum of
non-displaceable Lagrangian tori in $X$, whose construction is now in order.
Denote $\scr{A}_{GS}^2(\Z)$ by the free abelian group generated
by the horizontal $\&$ vertical Schubert cycles of codimension two
\begin{equation}\label{eq:basis}
\{\scr{D}_{i,i+1}^\textup{hor} \mid 1 \leq i \leq n-1 \} \cup \{\scr{D}_{j+1,j}^\textup{ver} \mid 1 \leq j \leq n-1 \}.
\end{equation}
Since $L \cap \scr{D}_{i,i+1}^\textup{hor} = \emptyset = L \cap \scr{D}_{j+1,j}^\textup{ver}$ for any $i, j$, the cap product of $\beta \in \pi_2(X,L)$ with any element thereof is well-defined.
Putting $\scr{A}_{GS}^2(\Lambda_0) := \scr{A}_{GS}^2(\Z) \otimes \Lambda_0,$
any element $\frak b \in \scr{A}_{GS}^2(\Lambda_0)$ can be expressed as
\begin{equation}\label{eq:frakb}
\frak b = \sum_{i=1}^{n-1} \frak b_{i,i+1}^\textup{hor}\scr{D}_{i,i+1}^\textup{hor}
+ \sum_{j=1}^{n-1} \frak b_{j+1,j}^\textup{ver}\scr{D}_{j+1,j}^\textup{ver}
\end{equation}
where $\frak b_{i,i+1}^\textup{hor}, \frak{b}_{j+1,j}^\textup{ver} \in \Lambda_0$.
We formally denote
\begin{equation}\label{formallydenote}
\beta \cap \frak b = \sum_{i=1}^{n-1} \frak b_{i,i+1}^\textup{hor} \left( \beta \cap \scr{D}_{i,i+1}^\textup{hor} \right)
+ \sum_{j=1}^{n-1} \frak b_{j+1,j}^\textup{ver} \left( \beta \cap \scr{D}_{j+1,j}^\textup{ver} \right).
\end{equation}
For simplicity, let us fix an enumeration $\{\scr{D}_j \mid j = 1,\cdots, B\}$ of the elements in \eqref{eq:basis} where $B = 2(n-2)$ and set $\frak{b}_j$ to be the coefficient corresponding to $\scr{D}_j$ in~\eqref{eq:frakb}.

The following is the statement of the counterpart of Theorem \ref{Foootoric2potenti}.

\begin{theorem}\label{mainthmappdendix} Let $\frak b \in \scr{A}_{GS}^2(\Lambda_0)$ and let $L_\varepsilon$ be a torus Lagrangian fiber in $X_\varepsilon$. Then the bulk-deformed potential function is written as
\begin{equation}\label{bulkdeformedpotential}
W^\frak{b} \left( L_\varepsilon ; b \right) = \sum_{\beta} n_\beta \cdot \exp \left( \beta \cap \frak{b} \right) \cdot \exp(\pa \beta \cap b) \, T^{\omega(\beta) / 2 \pi}.
\end{equation}
where the summation is taken over all homotopy classes in $\pi_2 (X_\varepsilon, L_\varepsilon)$ of Maslov index two.
\end{theorem}

The remaing part of this section is reserved for the proof of Theorem~\ref{mainthmappdendix}.

For a Lagrangian submanifold $L$ of $X$, denote by
$$
ev_i^{\text{\rm int}} \colon \mathcal M^{\text{\rm main}}_{k+1;\ell}(L,\beta) \to X
$$
the evaluation map at the $i$-th interior marked point for $i=1,\ldots,\ell$.
We put $\underline B = \{1,\ldots,B\}$ and denote the set of all maps $\text{\bf p} \colon
\{1,\ldots,\ell\} \to \underline B$ by
$Map(\ell,\underline B)$. We write $\vert\text{\bf
p}\vert = \ell$ if $\text{\bf p} \in Map(\ell,\underline B)$.
We define a fiber product
\begin{equation}
\mathcal M^{\text{\rm main}}_{k+1;\ell}(L,\beta;\text{\bf p})
= \mathcal M^{\text{\rm main}}_{k+1;\ell}(L,\beta)
{}_{(ev_1^{\text{int}},\ldots,ev_{\ell}^{\text{int}})}
\times_{X^{\ell}} \prod_{i=1}^{\ell} \scr{D}_{\text{\bf p}(i)}
\end{equation}
and consider the evaluation maps
$$
ev_i \colon \mathcal M^{\text{\rm main}}_{k+1;\ell}(L,\beta;\text{\bf p}) \to L \mbox {  by  }  ev_i((\Sigma,\varphi,\{z^+_i\}, \{z_i\})) = \varphi(z_i).
$$

Note the image of a Schubert horizontal or vertical cycle is (a union of) components of the toric divisor via the map $\phi^\prime_\varepsilon$ so that $\scr{D}_j$ can be regarded as the (union of) components of toric divisors corresponding to Schubert horizontal or vertical cycles. Also, by Proposition \ref{barM1} and Lemma~\ref{discsinXepsilon}, the holomorphic discs of Maslov index two intersect at the smooth locus of cycles. For a toric fiber $L_0$ in $X_0$, there is a system $\frak s = \{\frak s_{k+1,\beta;\textbf{\textup{p}}}\}$ of $T^n$-equivariant multi-sections on the moduli spaces $\mcal{M}_{k+1,l}(X_0, L_0; \beta; \textbf{\textup{p}})$ for all classes $\beta$ with $\mu(\beta) = 2$ in \cite[Lemma 6.5]{FOOOToric2}.

As in the correspondence in~\eqref{smoothcorrespondence0}, applying a smooth correspondence into
\begin{equation}\label{smoothcorrespondence}
	\xymatrix{
                              & {\mcal{M}_{k+1,l}(X_\varepsilon, L_\varepsilon; \beta; \textbf{\textup{p}})} \ar[dl]_{\bev_+} \ar[dr]^{\ev_0} &
      \\
 L^k & & L}
\end{equation}
we define
$$
\frak{q}_{k, \ell; \beta} \left( \textbf{p}; b^{\otimes k} \right) := \left(\ev_0\right)! \left(\bev_+^* (\pi_1^*b \otimes \cdots \otimes \pi_k^*b) \right).
$$
Since $X_\varepsilon$ is Fano, Lemma~\ref{discsinX0} and $L \cap \scr{D}_{i,i+1}^\textup{hor} = \emptyset = L \cap \scr{D}_{j+1,j}^\textup{ver}$ yield that $\mathcal M_{k+1;\ell}(X_\varepsilon, L_\varepsilon,\beta;\text{\bf p})$ empty if one of the following is satisfied.
\begin{equation}\label{threeconditionsforempty}
(1)\,\, \mu(\beta) < 0, \quad (2)\,\, \mu(\beta) = 0 \mbox{ and } \beta \neq 0, \quad (3)\,\, \beta = 0 \mbox{ and } l > 0.
\end{equation}
Because of the compatibility of the forgetful map forgetting the boundary marked points, see \cite[Section 5]{Fuk},
\begin{equation}\label{divisoraxiboundarymarked}
\frak{q}_{k, \ell; \beta} \left( \textbf{p}; b^{\otimes k} \right) = \frac{1}{k!} (\pa \beta \cap b)^k \cdot \frak{q}_{0, \ell; \beta} \left( \textbf{p}; 1 \right).
\end{equation}
Since any moduli spaces satisfying one of the conditions in~\eqref{threeconditionsforempty} are empty, $\frak{q}_{0, \ell; \beta}(\textbf{p}; 1)$ represents a cycle, yielding $L_\varepsilon$ is weakly unobstructed with respect to $\frak{b}$ in~\eqref{eq:frakb}. Passing to the canonical model \cite{FOOO, FOOOc}, we obtain
\begin{equation}\label{opengromovwittendef}
\frak{q}_{0, \ell; \beta}(\textbf{p}; 1) = n_{\beta} (\textbf{p}) \cdot \textup{PD}[L]
\end{equation}
for some $n_{\beta} (\textbf{p}) \in \Q$.  As a consequence, we obtain every $1$-cochain is a weak bounding cochain with respect to $\frak{b}$. In particular, the potential function with bulk is defined on $H^1(L_\varepsilon; \Lambda_+)$.

Under our situation, $n_\beta (\textbf{p})$ is well-defined. Especially when $\dim \scr{D}_{\bullet} = 2n -2$, $\mu(\beta) = 2$, we recall that this is
precisely the situation where the divisor axiom of the Gromov-Witten theory applies, see \cite[p.193]{CK} and \cite[Lemma 9.2]{FOOOToric2}. In particualr, we can calculate $n_\beta (\textbf{p})$ in the homology level and therefore
\begin{equation}\label{divisoraxiomopengromov}
n_\beta (\textbf{p}) = n_\beta \cdot \prod^{| \textbf{p} |}_{i=1} \left(\beta \cap \scr{D}_{\textbf{p}(i)} \right).
\end{equation}
Following \cite[Section 7]{FOOOToric2} and using~\eqref{divisoraxiboundarymarked},~\eqref{opengromovwittendef},~\eqref{divisoraxiomopengromov} and~\eqref{formallydenote}, we obtain that
\begin{align*}
\sum_{k=0}^\infty \frak{m}^\frak{b}_k \left( b^{\otimes k} \right) &:= \sum_{k=0}^\infty \sum_{\ell=0}^\infty \sum_{\beta; \mu(\beta) = 2} \frac{1}{\ell!} \,\, \frak{q}_{k, \ell; \beta} \left(\frak{b}^{\otimes \ell}; b^{\otimes k} \right) T^{\omega(\beta) / 2 \pi} \\
&= \sum_{\ell=0}^\infty \sum_{\textbf{p}; |\textup{p} | = \ell} \sum_{\beta; \mu(\beta) = 2}  \exp \left( \pa \beta \cap b \right) \cdot \frac{1}{\ell!} \,\, \frak{b}^\textbf{p} \, \frak{q}_{0, \ell; \beta} \left(\textbf{p}; 1 \right) T^{\omega(\beta) / 2 \pi} \\
&= \sum_{\beta; \mu(\beta) = 2}  \left( \sum_{\ell=0}^\infty \sum_{\textbf{p}; |\textup{p} | = \ell}  \exp \left( \pa \beta \cap b \right) \cdot \frac{1}{\ell!} \,\, \frak{b}^\textbf{p} \, n_\beta(\textbf{p}) \right) T^{\omega(\beta) / 2 \pi} \cdot PD[L] \\
&= \sum_{\beta; \mu(\beta) = 2} n_\beta \cdot \left( \sum_{\ell=0}^\infty \frac{1}{\ell!}  \sum_{\textbf{p}; |\textup{p} | = \ell}  \,\, \prod^{| \textbf{p} |}_{i=1} \frak{b}_{\textbf{p}(i)} \left(  \beta \cap \scr{D}_{\textbf{p}(i)} \right)  \right) \cdot \exp(\pa \beta \cap b) \,\cdot T^{\omega(\beta) / 2 \pi}\cdot PD[L]\\
&= \sum_{\beta; \mu(\beta) = 2} n_\beta \cdot \exp \left( \beta \cap \frak{b} \right) \cdot \exp(\pa \beta \cap b) \,\cdot T^{\omega(\beta) / 2 \pi}\cdot PD[L]
\end{align*}
where $\frak{b}^\textbf{p} = \prod_{i=1}^\ell \frak{b}_{\textbf{p}(i)}$. Finally, incorporating with deformation of non-unitary flat line bundle by Cho \cite{Cho}, we extend the domain of the bulk-deformed potential to $H^1(L_\varepsilon; \Lambda_0)$. This completes the proof of Theorem~\ref{mainthmappdendix}.

\end{document}